\documentclass[a4paper,12pt]{amsart}
\usepackage{amsmath,amssymb,amsxtra,amsthm,amsrefs}
\usepackage{fourier}
\usepackage{mathdots}
\usepackage[scaled=0.875]{helvet}

\usepackage[all]{xy}
\let\comp=\circ
\let\capprod=\cap
\newcommand*\fixitem {\item[]%
\refstepcounter{enumi}\hskip-\leftmargin\labelenumi\hskip\labelsep}
\makeatletter
\newif\if@mainmatter \@mainmattertrue
\newtheoremstyle{mytheorem}  
  {3pt}   
  {3pt}   
  {\itshape}  
  {0pt}       
  {\bfseries} 
  {.}         
  {5pt plus 1pt minus 1pt} 
  {}          
\newtheoremstyle{mydefinition}  
  {3pt}   
  {3pt}   
  {\normalfont}  
  {0pt}       
  {\bfseries} 
  {.}         
  {5pt plus 1pt minus 1pt} 
  {}          
\renewenvironment{proof}[1][\proofname]{\par\noindent
  \normalfont
  \topsep6\p@\@plus6\p@ \trivlist
  \item[\hskip\labelsep\bfseries
    #1\@addpunct{.}]\ignorespaces
}{%
  \qed\endtrivlist
}
\usepackage{fancyhdr}
\def\@setauthors{%
  \begingroup
  \def\thanks{\protect\thanks@warning}%
  \trivlist
  \centering\normalsize \@topsep30\p@\relax
  \advance\@topsep by -\baselineskip
  \item\relax
  \author@andify\authors
  \def\\{\protect\linebreak}%
  \authors%
  \ifx\@empty\contribs
  \else
    ,\penalty-3 \space \@setcontribs
    \@closetoccontribs
  \fi
  \endtrivlist
  \endgroup
}
\def\@settitle{\begin{center}%
  \baselineskip14\p@\relax
    {\large \bfseries\mathversion{bold}
  \@title}
  \end{center}%
}
\usepackage{titlesec}
\titleformat{\section}
  {\normalfont\fontsize{12}{15}\bfseries}{\S\thesection.}{1mm}{}
\titleformat{\subsection}
  {\normalfont\fontsize{12}{15}\bfseries}{\thesubsection.}{1mm}{}
\titleformat{\subsubsection}
  {\normalfont\fontsize{12}{15}\bfseries}{\thesubsubsection.}{1mm}{}

\renewenvironment{abstract}{%
  \ifx\maketitle\relax
    \ClassWarning{\@classname}{Abstract should precede
      \protect\maketitle\space in AMS document classes; reported}%
  \fi
  \global\setbox\abstractbox=\vtop \bgroup
    \normalfont\Small
    \list{}{\labelwidth\z@
      \leftmargin3pc \rightmargin\leftmargin
      \listparindent\normalparindent \itemindent\z@
      \parsep\z@ \@plus\p@
      
    }%
    \item[\hskip\labelsep\upshape\bfseries\abstractname.]%
}{%
  \endlist\egroup
  \ifx\@setabstract\relax \@setabstracta \fi
}

    \@addtoreset{equation}{section}
\newcommand\ackname{Acknowledgements}
\if@titlepage
\newenvironment{acknowledgements}{%
      \titlepage
      \null\vfil
      \@beginparpenalty\@lowpenalty
      \begin{center}%
        \bfseries \ackname
        \@endparpenalty\@M
      \end{center}}%
     {\par\vfil\null\endtitlepage}
\else
  
\fi
    \def\@citex[#1]#2{\leavevmode
      \let\@citea\@empty
      \@cite{\bfseries\@for\@citeb:=#2\do
        {\@citea\def\@citea{,\penalty\@m\ }%
         \edef\@citeb{\expandafter\@firstofone\@citeb\@empty}%
         \if@filesw\immediate\write\@auxout{\string\citation{\@citeb}}\fi
         \@ifundefined{b@\@citeb}{\hbox{\reset@font\bfseries ?}%
           \G@refundefinedtrue
           \@latex@warning
             {Citation `\@citeb' on page \thepage \space undefined}}%
           {\@cite@ofmt{\csname b@\@citeb\endcsname}}}}{{\rm#1}}}
\makeatother
\usepackage[dvipdfmx,%
bookmarks=true,%
bookmarksnumbered=true,%
colorlinks=false,%
setpagesize=false,%
linkcolor=blue,%
pdftitle={pseudo H-type Lie algebras},%
pdfsubject={graded Lie algebra},%
pdfkeywords={TeX; dvipdfmx; hyperref; color;}
]{hyperref}

\makeatletter
\def\@setauthors{%
  \begingroup
  \def\thanks{\protect\thanks@warning}%
  \trivlist
  \centering\normalsize \@topsep30\p@\relax
  \advance\@topsep by -\baselineskip
  \item\relax
  \author@andify\authors
  \def\\{\protect\linebreak}%
  \authors%
  \ifx\@empty\contribs
  \else
    ,\penalty-3 \space \@setcontribs
    \@closetoccontribs
  \fi
  \endtrivlist
  \endgroup
}
\def\@settitle{\begin{center}%
  \baselineskip14\p@\relax
    {\large \bfseries\mathversion{bold}
  \@title}
  \end{center}%
}

    \@addtoreset{equation}{section}
\makeatother 
\newcommand{\Ker}{\operatorname{Ker}} 
\newcommand{\ad}{\operatorname{ad}} 
\newcommand{\rea}{\operatorname{Re}} 
\newcommand{\sgn}{\operatorname{sgn}} 
\newcommand{\pr}{\operatorname{pr}} 
 
\newcommand{\tr}{\operatorname{tr}} 
\newcommand{\ima}{\operatorname{Im}} 
\newcommand{\Cl}{\operatorname{Cl}} 
\newcommand{\ric}{\operatorname{ric}} 
\newcommand{\Ric}{\operatorname{Ric}} 
\newcommand{\scc}{\operatorname{sc}} 
\newcommand{\End}{\operatorname{End}} 
\newcommand{\scalar}[2]{\langle{#1}\mid#2\rangle} 
\newcommand{\Iso}{\operatorname{Iso}} 
\newcommand{\aut}{\operatorname{aut}} 
\newcommand{\Aut}{\operatorname{Aut}} 
\newcommand{\Spl}{\operatorname{spl}}

\newcommand{\Der}{\operatorname{Der}}
\newcommand{\transpose}[1]{\,{}^t\!#1}
\newcommand{\inph}[2]{\langle#1\,|\,#2\rangle}
\newcommand{\innerproduct}{\inph{\cdot}{\cdot}}
  
\newcommand{\gla}[1]{\mathfrak#1=\bigoplus\limits_{p\in\mathbb Z}\mathfrak#1_p} 
 
\newcommand{\diag}{\operatorname{diag}}  
\newcommand{\pcgla}[2][m]{\mathfrak #2(\mathfrak #1)=
\bigoplus\limits_{p\in\mathbb Z}\mathfrak #2(\mathfrak #1)_p}
\renewcommand{\labelenumi}{\textup{(\arabic{enumi})}\,} 
\renewcommand{\labelenumii}{(\roman{enumii})\,} 
\makeatletter
\newcommand{\subsubsubsection}{\@startsection{paragraph}{4}{\z@}%
    {1.0\Cvs \@plus.5\Cdp \@minus.2\Cdp}%
    {.1\Cvs \@plus.3\Cdp}%
    {\reset@font\rmfamily\normalsize\bfseries}
  }
\makeatother
\theoremstyle{mytheorem}
\newtheorem{theorem}{Theorem}[section]
\newtheorem{lemma}{Lemma}[section]
\newtheorem{proposition}{Proposition}[section]
\theoremstyle{mydefinition} 
\newtheorem{remark}{Remark}[section] 
 
\newtheorem{example}{Example}[section] 
\newtheorem{corollary}{Corollary}[section] 

\title{
On the prolongation of a conformal pseudo-subriemannian fundamental graded Lie algebra associated with a pseudo-H-type Lie algebra}
\author{Tomoaki Yatsui}
\address[Tomoaki Yatsui]{Masakae 1-9-2, Otaru, 047-0003, Japan}
\email{yatsui.tomoaki@gmail.com}
\begin{document}
\maketitle

\begin{abstract}
A pseudo H-type Lie algebra naturally gives rise to a conformal 
pseudo-subriemannian fundamental graded Lie algebras. 
In this paper we investigate the prolongations of the associated fundamental graded Lie algebra and the associated conformal pseudo-subriemannian fundamental graded Lie algebra. 
In particular, we show that the prolongation of the associated conformal 
pseudo-subriemannian fundamental graded Lie algebra coincides with that of the associated fundamental graded Lie algebra under some assumptions. 
\end{abstract}
\section{Introduction}
In \cite{Kap80:1} A.~Kaplan introduced H-type Lie algebras, 
which belong to an important class of metric metabelian Lie algebras. 
Furthermore a large class of \textit{pseudo H-type Lie algebras}, which are obtained by replacing 
the inner product to a scalar product, appeared in P.~Citatti \cite{Cia00:1}. 
We will give the precise definition of a pseudo H-type Lie algebra below.  

Let $\mathfrak n$ be a finite dimensional 2-step nilpotent real Lie algebras, 
that is, 
$\mathfrak n$ is a finite dimensional real Lie algebra satisfying 
$[\mathfrak n,\mathfrak n]\ne0$ and 
$[\mathfrak n,[\mathfrak n,\mathfrak n]]=0$. 
Let $\innerproduct$ be a scalar product on $\mathfrak n$ 
such that the center $\mathfrak n_{-2}$ of $\mathfrak n$ is a 
nondegenerate subspace of $(\mathfrak n,\innerproduct)$. 
Here a scalar product on $\mathfrak n$ means a nondegenerate symmetric 
bilinear form on $\mathfrak n$. 
Let $\mathfrak n_{-1}$ be the orthogonal complement of $\mathfrak n_{-2}$ with respect to $\innerproduct$. 
The pair $(\mathfrak n,\innerproduct)$ is called a 
\textit{pseudo H-type Lie algebra}  if for any $z\in \mathfrak n_{-2}$ the endomorphism $J_z$ of $\mathfrak n_{-1}$ defined by 
$\inph{J_z(x)}{y}=\inph{z}{[x,y]}$ $(x,y\in\mathfrak n_{-1})$ 
satisfies the Clifford condition $J_z^2=-\inph{z}{z}1_{\mathfrak n_{-1}}$, 
where $1_{\mathfrak n_{-1}}$ is the identity transformation of 
$\mathfrak n_{-1}$. 
In particular, if $\innerproduct$ is positive definite, then 
$(\mathfrak n,\innerproduct)$ is simply called an H-type Lie algebra. 

Let $(\mathfrak n,\innerproduct)$ be a pseudo H-type Lie algebra. 
Then $\mathfrak n=\mathfrak n_{-2}\oplus \mathfrak n_{-1}$ becomes a 
nondegenerate fundamental graded Lie algebra of the second kind, 
which is called associated with  $(\mathfrak n,\innerproduct)$. 

Now we explain the notion of a fundamental graded Lie algebra and its prolongation briefly. 
A finite dimensional graded Lie algebra (GLA) 
$\mathfrak m=\bigoplus\limits_{p<0}\mathfrak g_p$ is called a 
\textit{fundamental graded Lie algebra} (FGLA) of 
the $\mu$-th kind if the following conditions hold: 
(i) $\mathfrak g_{-1}\ne0$, 
and $\mathfrak m$ is generated by $\mathfrak g_{-1}$; 
(ii) $\mathfrak g_{-\mu}\ne0$ and $\mathfrak g_p=0$ for all $p<-\mu$, 
where $\mu$ is a positive integer. 
Furthermore an FGLA $\mathfrak m=\bigoplus\limits_{p<0}\mathfrak g_p$ is 
called nondegenerate if for $x\in\mathfrak g_{-1}$, $[x,\mathfrak g_{-1}]=0$ 
implies $x=0$. 
For a given FGLA $\mathfrak m=\bigoplus\limits_{p<0}\mathfrak g_p$ 
there exists a GLA 
$\pcgla g$ satisfying the following three conditions (P1)--(P3): 
(P1) The negative part $\mathfrak g(\mathfrak m)_-
=\bigoplus\limits_{p<0}\mathfrak g(\mathfrak m)_p$ of 
$\pcgla g$ coincides with a given FGLA $\mathfrak m$ as a GLA; 
(P2) For $x\in \mathfrak g(\mathfrak m)_p$ $(p\geqq0)$, $[x,\mathfrak g_{-1}]=0$ 
implies $x=0$; 
(P3) $\pcgla g$ is maximum among GLAs satisfying the conditions (P1) and (P2) 
above. 
The GLA $\pcgla g$ is called the prolongation of the FGLA 
$\mathfrak m$. 
Given the prolongation $\pcgla g$ of an FGLA $\mathfrak m$, 
an element $E$ of $\mathfrak g(\mathfrak m)_0$ is called the characteristic element of 
$\pcgla g$ if $[E,x]=px$ for all $x\in\mathfrak g(\mathfrak m)_p$ and $p\in\mathbb Z$. 
Also $\ad(\mathfrak g(\mathfrak m)_0)|\mathfrak m$ is a subalgebra  of the derivation algebra $\Der(\mathfrak m)$ of $\mathfrak m$ isomorphic to $\mathfrak g(\mathfrak m)_0$; 
we identify it with $\mathfrak g(\mathfrak m)_0$ in what follows, 
so that $D\in\mathfrak g(\mathfrak m)_0$ is identified with $\ad(D)|\mathfrak m$ 
(For the details of FGLAs and a construction of the 
prolongation, see \cite{Tan70:1}*{\S5}). 
Note that the prolongation of a nondegenerate FGLA 
$\mathfrak m=\bigoplus\limits_{p<0}\mathfrak g_p$ of the second kind is of infinite dimension if $\dim \mathfrak g_{-2}\leqq2$. 

The first purpose of this paper is to investigate 
the prolongation of the FGLA associated with a pseudo H-type Lie algebra. 
In \cite{AS14:1} A.~Altomani and A.~Santi classified  
the prolongation of extended translation algebras provided that it is nontrivial. 
Since pseudo H-type Lie algebras are contained in 
the class of extended translation algebras, we can apply the results of \cite{AS14:1,AS14:2} to our study. 
In \S4 we obtain the following theorem (Theorem \ref{th42}): 
Let $(\mathfrak n,\innerproduct)$ be a pseudo H-type Lie algebra 
and $\pcgla[n]{g}$ be the prolongation of the FGLA 
associated with $(\mathfrak n,\innerproduct)$. 
Assume that $\dim\mathfrak n_{-2}\geqq3$ and $\mathfrak g(\mathfrak n)_1\ne0$. 
Then 
(1) $\pcgla[n]{g}$ is isomorphic to a simple GLA (abbreviated to SGLA), whose complexification is also simple;  
(2) In addition, if $\sgn(\innerproduct_{-2}\ne(1,3),(3,1)$, then 
$(\mathfrak n,\innerproduct)$ is isomorphic to a 
comH-type Lie algebra, which is appeared below (also see \S4). 

We next give the notion of a conformal pseudo-subriemannian FGLA and 
its prolongation. 
We say that the pair $(\mathfrak m,[g])$ of a real FGLA 
$\mathfrak m=\bigoplus\limits_{p<0}\mathfrak g_p$  of 
the $\mu$-th kind $(\mu\geqq2)$ and the conformal class $[g]$ of 
a scalar product $g$ on $\mathfrak g_{-1}$ is a 
\textit{conformal pseudo-subriemannian FGLA} (CPSF). 
For a given CPSF $(\mathfrak m,[g])$ 
let $\pcgla g$ be the prolongation of $\mathfrak m$, and let 
 $\mathfrak g_0$ be the subalgebra of $\mathfrak g(\mathfrak m)_0$ consisting 
of all elements $D$ of $\mathfrak g(\mathfrak m)_0$ 
such that $\ad(D)|\mathfrak g_{-1}\in\mathfrak{co}(\mathfrak g_{-1},g)$.  
We define a sequence $(\mathfrak g_p)_{p\geqq1}$ inductively as follows: 
$\ell$ being a positive integer, suppose that we defined $\mathfrak g_1,\dots,\mathfrak g_{\ell-1}$ as subspaces of 
$\mathfrak g(\mathfrak m)_1,\dots,\mathfrak g(\mathfrak m)_{\ell-1}$ 
respectively, in such a way that $[\mathfrak g_p,\mathfrak g_r]\subset\mathfrak g_{p+r}$ $(0<p<\ell,r<0)$. Then we define $\mathfrak g_\ell$ to be the subspace of 
$\mathfrak g(\mathfrak m)_\ell$ consisting of all the elements $D$ of 
$\mathfrak g(\mathfrak m)_\ell$ such that 
$[D,\mathfrak g_r]\subset\mathfrak g_{\ell+r}$ $(r<0)$. 
If we put $\gla g$, then it becomes a graded subalgebra of $\pcgla g$, 
which is called the prolongation of $(\mathfrak m,\mathfrak g_0)$. 
The prolongation of $(\mathfrak m,\mathfrak g_0)$ is also called that of the CPSF $(\mathfrak m,[g])$. 
The prolongation $\gla g$ of the CPSF $(\mathfrak m,[g])$ is finite dimensional. 
If $\gla g$ is semisimple, then the CPSF $(\mathfrak m,[g])$ is said to 
be of semisimple type. Note that the prolongation of CPSF $(\mathfrak m,[g])$ of semisimple type is an SGLA and the complexification is also simple. 
In the previous paper \cite{Yat18:1} we classified the prolongations of 
CPSFs of semisimple type. 

Let $(\mathfrak n,\innerproduct)$ be a pseudo H-type Lie algebra. 
The pair $(\mathfrak n,[\innerproduct_{-1}])$ becomes a CPSF, which is called associated with $(\mathfrak n,\innerproduct)$. 
Here we denote by $\innerproduct_{k}$ 
the restriction of $\innerproduct$ to $\mathfrak n_k$. 
The second purpose of this paper is to classify the prolongation of CPSFs 
associated with pseudo $H$-type Lie algebras. 

In \cite{KS17:1} A.~Kaplan and M.~Sublis introduced the notion of a 
divH-type Lie algebra (or a Lie algebra of type divH) induced from finite dimensional real division 
algebras and classified the finite dimensional real SGLAs 
whose negative parts are isomorphic to some divH-type Lie algebra. 
In \cite{KS16:1} they also proved that the prolongation of the FGLA associated 
with an H-type Lie algebra is not trivial if and only if it is a divH-type Lie algebra. 
In \S3, motivated by the studies in \cite{KS17:1} and \cite{Gom96:1}, 
we define \textit{comH-type Lie algebras} 
induced from finite dimensional real composition algebras,  
which are decomposed into three classes (comH-type Lie algebras of the first, 
the second and the third classes). 
Note that our construction is slightly different from the definition in \cite{KS17:1}. 

The second purpose of this paper is to determine the prolongations of the FGLAs associated with comH-type Lie algebras by an elementary method. 
It is known that a pseudo H-type Lie algebra satisfying the $J^2$-condition 
becomes a comH-type Lie algebra of the first class, and vice versa (cf.\cite{MKM18:1}). 
In \S6 we prove that a comH-type Lie algebra 
satisfies the $J^2$-condition if and only if the associated CPSF is of semisimple type  (Theorem \ref{th51}). 

In \S6 we show the following theorem (Theorem \ref{th62}): 
Let $(\mathfrak n,\innerproduct)$ be a pseudo $H$-type Lie algebra 
and $\gla g$ be the prolongation of a CPSF $(\mathfrak n,[\innerproduct_{-1})$. 
If $\mathfrak g_1\ne0$, then 
it is isomorphic to a comH-type Lie algebra of the first class and 
$\gla g$ is an SGLA. 
In addition if $\dim \mathfrak n_{-2}\geqq3$, then 
$\gla g$ coincides with the prolongation of $\mathfrak n$. 

In the final section we state the fact that the pseudo-Iwasawa extension of a pseudo H-type Lie algebra is Einstein.  
\section*{Notation and Conventions}
Throughout the paper the following notation is used. 
\subsection*{Matrices}
\begin{enumerate}
\item The $n\times n$ identity matrix (resp. zero matrix) is written $1_n$ 
(resp. $0_n$). Also 
for a set $X$ we denote by $1_X$ the identity transformation of $X$. 
\item The block diagonal matrix (resp. The anti-diagonal matrix) with the diagonal blocks (resp. the anti-diagonal blocks) $A_1,\dots,A_p$ (in that order) 
will be denoted by 
$A_1\oplus \cdots \oplus A_p$ (resp. 
$A_1\ominus \cdots \ominus A_p$).
$$A_1\oplus \cdots \oplus A_p
=\begin{bmatrix}
A_1 & 0 & \cdots & 0\\
0 & A_2 & \cdots & 0 \\
\vdots & \vdots &\ddots & \vdots\\ 
0 & 0 & \cdots & A_p
\end{bmatrix},\quad 
A_1\ominus \cdots \ominus A_p
=\begin{bmatrix}
0 & \cdots  & 0 & A_1\\
0 & \cdots & A_2 & 0 \\
\vdots & \iddots &\vdots & \vdots\\ 
A_p & \cdots  & 0 & 0
\end{bmatrix}.
$$
\item For positive integers $r,s$ we define a matrix $1_{r,s}$ as follows: 
$$1_{r,s}=1_r\oplus(-1_s)$$
\item We denote by $K_m$ the sip matrix of the size $m$, i.e., 
$K_m$ is an $m\times m$ matrix with the $(i,j)$-component $\delta_{i,m-j+1}$. 
Then $K_m$ is a real symmetric matrix with $K_m^2=1_m$. 
\item We define an $n\times n$ symmetric real matrix $S_{p,q}$ as follows: 
$$S_{p,q}=K_p\ominus 1_q\ominus K_p=
\begin{bmatrix}
0 & 0 & K_p \\
0 & 1_q & 0 \\
K_p & 0 & 0 \\
\end{bmatrix}\qquad (p\geqq1,q\geqq0,2p+q=n).$$
Here the center column and the center row of $S_{p,q}$ should be deleted 
when $q=0$. Also we set $S_{0,q}=1_q$. 
Then $S_{p,q}$ is a symmetric real matrix with signature $(p+q,p)$. 
\item 
We put  
$$J_{2m}=K_m\ominus (-K_m)
,\quad 
I_{2m}=1_m\ominus (-1_m)
.$$
Then $J_{2m}$ and $I_{2m}$ are skew-symmetric matrices. Clearly 
$J_{2m}^2=I_{2m}^2=-1_{2m}$ and 
$K_{2m}J_{2m}=-J_{2m}K_{2m}$. 
\item For a matrix $A$ we denote by $\transpose{A}$ (resp. $A^*$) the transposed matrix (resp. the  conjugate transposed matrix) of $A$. 
\end{enumerate}
\subsection*{Composition algebras}
\begin{enumerate}
\item Blackboard bold is used for the standard systems 
$\mathbb Z$ (the ring of integers), $\mathbb R$ (real numbers),
$\mathbb C$ (complex numbers), $\mathbb C'$ (split complex numbers),
the real division rings $\mathbb H$ (Hamilton's quaternions), 
$\mathbb H'$ (split quaternions), 
$\mathbb O$ (Cayley's [nonassociative] octonions) 
and $\mathbb O'$ (split octonions) 
(cf. \cite{Har90:1}*{Ch.~6}, \cite{Gom96:1}*{2.1}). 
\item 
For $\mathbb K=\mathbb C$, $\mathbb C'$, $\mathbb H$, $\mathbb H'$, 
$\mathbb O$ or $\mathbb O'$, we set 
$\ima \mathbb K=\{~z\in\mathbb K:\rea z=0~\}$.  
\end{enumerate}
\subsection*{Lie algebras}
\begin{enumerate}
\item $\mathrm{(AI)}_\ell$ is the Satake diagram of $\mathfrak{sl}(\ell+1,\mathbb R)$. 

\begin{center}
\begin{minipage}{5cm}
\begin{xy}
(-10,0) *{\textrm{(AI)}_{^\ell}:}="K", 
(0,0) *{\circ}="A"*++!D{{\scriptstyle 1}},
(10,0) *{\circ}="B"*++!D{{\scriptstyle 2}},
(30,0) *{\circ}="C"*++!D{{\scriptstyle \ell-1}},
(40,0) *{\circ}="D"*++!D{{\scriptstyle \ell}},
\ar @{-} "A";"B"
\ar @{.} "B";"C"
\ar @{-} "C";"D"
\end{xy}
\end{minipage}
\end{center}

\item $\mathrm{(AII)}_\ell$ is the Satake diagram of $\mathfrak{sl}(m,\mathbb H)$ 
$(\ell=2m-1)$. 
\begin{center}
\begin{minipage}{5cm}
\begin{xy}
(-20,0) *{\textrm{(AII)}_{^\ell}:}="K", 
(-10,0) *{\bullet}="A"*++!D{{\scriptstyle 1}},
(0,0) *{\circ}="B"*++!D{{\scriptstyle 2}},
(10,0) *{\bullet}="C"*++!D{{\scriptstyle 3}},
(30,0) *{\circ}="D"*++!D{{\scriptstyle \ell-1}},
(40,0) *{\bullet}="E"*++!D{{\scriptstyle \ell}},
\ar @{-} "A";"B"
\ar @{-} "B";"C"
\ar @{.} "C";"D"
\ar @{-} "D";"E"
\end{xy}
\end{minipage}
\end{center}
\item $\mathrm{(AIIIa)}_{\ell,p}$ is the Satake diagram of $\mathfrak{su}(\ell-p+1,p)$ 
$(2\leqq p\leqq \ell/2)$. 
\begin{center}
\begin{minipage}{5cm}
\begin{xy}
(-20,-7) *{\textrm{(AIIIa)}_{\ell,p}:}="K", 
(0,0) *{\circ}="A"*++!D{{\scriptstyle 1}},
(10,0) *{\circ}="B"*++!D{{\scriptstyle 2}},
(30,0) *{\circ}="C"*++!D{{\scriptstyle p}},
(40,0) *{\bullet}="D"*++!D{{\scriptstyle p+1}},
(40,-5) *{\bullet}="E",
(40,-10) *{\bullet}="F",
(40,-15) *{\bullet}="G"*++!U{{\scriptstyle \ell-p}},
(30,-15) *{\circ}="H"*++!U{{\scriptstyle \ell-p+1}},
(10,-15) *{\circ}="I"*++!U{{\scriptstyle \ell-1}},
(0,-15) *{\circ}="J"*++!U{{\scriptstyle \ell}},
\ar @{-} "A";"B"
\ar @{.} "B";"C"
\ar @{-} "C";"D"
\ar @{-} "D";"E"
\ar @{.} "E";"F"
\ar @{-} "F";"G"
\ar @{-} "G";"H"
\ar @{.} "H";"I"
\ar @{-} "I";"J"
\ar @{<->} @/_3mm/ "A";"J"
\ar @{<->} @/_3mm/ "B";"I"
\ar @{<->} @/_3mm/ "C";"H"
\end{xy}
\end{minipage}
\end{center}
\item $\mathrm{(AIIIb)}_\ell$ is the Satake diagram of $\mathfrak{su}(p,p)$ $(\ell=2p-1)$. 
\begin{center}
\begin{minipage}{5cm}
\begin{xy}
(-25,-5) *{\textrm{(AIIIb)}_{\ell}:}="K", 
(0,0) *{\circ}="A"*++!D{{\scriptstyle 1}},
(10,0) *{\circ}="B"*++!D{{\scriptstyle 2}},
(30,0) *{\circ}="C"*++!D{{\scriptstyle p-1}},
(40,-5) *{\circ}="D"*++!D{{\scriptstyle p}},
(30,-10) *{\circ}="E"*++!U{{\scriptstyle p+1}},
(10,-10) *{\circ}="F"*++!U{{\scriptstyle \ell-1}},
(0,-10) *{\circ}="G"*++!U{{\scriptstyle \ell}},
\ar @{-} "A";"B"
\ar @{.} "B";"C"
\ar @{-} "C";"D"
\ar @{-} "D";"E"
\ar @{.} "E";"F"
\ar @{-} "F";"G"
\ar @{<->} @/_3mm/ "A";"G"
\ar @{<->} @/_3mm/ "B";"F"
\ar @{<->} @/_3mm/ "C";"E"
\end{xy}
\end{minipage}
\end{center}

\item $\mathrm{(AIV)}_{\ell}$ is the Satake diagram of $\mathfrak{su}(\ell,1)$. 
\begin{center}
\begin{minipage}{5cm}
\begin{xy}
(-10,0) *{\textrm{(AIV)}_\ell:}="K", 
(0,0) *{\circ}="A"*++!D{{\scriptstyle 1}},(10,0) *{\bullet}="B",(30,0) *{\bullet}="C",
(40,0) *{\circ}="D"*++!D{{\scriptstyle \ell}},
\ar @{-} "A";"B"
\ar @{.} "B";"C"
\ar @{-} "C";"D"
\ar @{<->} @/^6mm/ "A";"D"
\end{xy}
\end{minipage}
\end{center}

\item $\mathrm{(CI)}_\ell$ is the Satake diagram of $\mathfrak{sp}(\ell,\mathbb R)$. 

\begin{center}
\begin{minipage}{5cm}
\begin{xy}
(-10,0) *{\textrm{(CI)}_{\ell}:}="K", 
(0,0) *{\circ}="A"*++!D{{\scriptstyle 1}},
(10,0) *{\circ}="B"*++!D{{\scriptstyle 2}},
(30,0) *{\circ}="C"*++!D{{\scriptstyle \ell-1}},
(40,0) *{\circ}="D"*++!D{{\scriptstyle \ell}},
\ar @{-} "A";"B"
\ar @{.} "B";"C"
\ar @{=>} "D";"C"
\end{xy}
\end{minipage}
\end{center}
\medskip

\item $\mathrm{(CIIa)}_{\ell,p}$ is the Satake diagram of $\mathfrak{sp}(\ell-p,p)$. 
\begin{center}
\begin{minipage}{5cm}
\begin{xy}
(-13,0) *{\textrm{(CIIa)}_{\ell,p}:}="K", 
(0,0) *{\bullet}="A"*++!D{{\scriptstyle 1}},
(10,0) *{\circ}="B"*++!D{{\scriptstyle 2}},
(20,0) *{\bullet}="C"*++!D{{\scriptstyle 3}},
(35,0) *{\circ}="D"*++!D{{\scriptstyle 2p}},
(45,0) *{\bullet}="E"*++!D{{\scriptstyle 2p+1}},
(60,0) *{\bullet}="F"*++!D{{\scriptstyle \ell-1}},
(70,0) *{\bullet}="G"*++!D{{\scriptstyle \ell}},
\ar @{-} "A";"B"
\ar @{-} "B";"C"
\ar @{.} "C";"D"
\ar @{-} "D";"E"
\ar @{.} "E";"F"
\ar @{=>} "G";"F"
\end{xy}
\end{minipage}
\end{center}

\item $\mathrm{(CIIb)}_{\ell}$ is the Satake diagram of $\mathfrak{sp}(p,p)$ $(\ell=2p)$. 
\begin{center}
\begin{minipage}{5cm}
\begin{xy}
(-18,0) *{\textrm{(CIIb)}_{\ell}:}="K", 
(0,0) *{\bullet}="A"*++!D{{\scriptstyle 1}},
(10,0) *{\circ}="B"*++!D{{\scriptstyle 2}},
(20,0) *{\bullet}="C"*++!D{{\scriptstyle 3}},
(40,0) *{\circ}="D"*++!D{{\scriptstyle 2p-2}},
(50,0) *{\bullet}="E"*++!D{{\scriptstyle 2p-1}},
(60,0) *{\circ}="F"*++!D{{\scriptstyle 2p}},
\ar @{-} "A";"B"
\ar @{-} "B";"C"
\ar @{.} "C";"D"
\ar @{-} "D";"E"
\ar @{=>} "F";"E"
\end{xy}
\end{minipage}
\end{center}

\item For the description of finite dimensional SGLAs, we use the notations in 
\cite{Yam93:1}*{\S3}. 
\end{enumerate}

\section{Pseudo H-type Lie algebras}
Following \cite{Cia00:1} we define pseudo H-type Lie algebras. 
Let $\mathfrak n$ be a finite dimensional 2-step nilpotent real Lie algebra 
equipped with a nondegenerate symmetric bilinear form 
$\inph{\cdot}{\cdot}$ on $\mathfrak n$. 
The pair $(\mathfrak n,\inph{\cdot}{\cdot})$ is called a pseudo H-type Lie algebra if the following conditions hold: 
\begin{enumerate}
\renewcommand{\labelenumi}{(H.\arabic{enumi}) }
\item 
The restriction  of $\innerproduct$ to 
the center $\mathfrak n_{-2}$ of $\mathfrak n$ is nondegenerate. 
\item Let $\mathfrak n_{-1}$ be the orthogonal complement of the center $\mathfrak n_{-2}$ of $\mathfrak n$ with respect to $\innerproduct$. 
For any $z\in \mathfrak n_{-2}$ the endomorphism  $J_{z}$ of 
$\mathfrak n_{-1}$ defined by 
\begin{equation}\inph{J_z(x)}{y} =\inph{z}{[x,y]} \qquad 
x,y\in \mathfrak n_{-1},  
\label{eq22}
\end{equation}
satisfies the following condition 
\begin{equation}
J_z^2=-\inph{z}{z} 1_{\mathfrak n_{-1}},
\label{eq23}
\end{equation}  
where $1_{\mathfrak n_{-1}}$ is the identity transformation of $\mathfrak n_{-1}$. 
\end{enumerate}
The condition \eqref{eq23} is called the Clifford condition. 
In particular if $\innerproduct$ is positive definite, 
then $(\mathfrak n,\innerproduct)$ is simply called an H-type Lie algebra. 
Given a pseudo H-type Lie algebra $(\mathfrak n,\innerproduct)$ 
we can easily see that: 
\begin{enumerate}
\renewcommand{\labelenumi}{(\roman{enumi})\ }
\item For any $z\in \mathfrak n_{-2}$ 
the linear mapping $J_z$ is skew-symmetric; 
\item $\mathfrak n=\mathfrak n_{-1}\oplus \mathfrak n_{-2}$ is 
a non-degenerate FGLA of the second kind. 
\end{enumerate}
The FGLA $\mathfrak n=\mathfrak n_{-1}\oplus \mathfrak n_{-2}$ 
is called associated with 
the pseudo H-type Lie algebra $(\mathfrak n,\langle\cdot\mid\cdot\rangle)$.  
The pair $(\mathfrak n=\mathfrak n_{-1}\oplus \mathfrak n_{-2},
[\innerproduct_{-1}])$ becomes a conformal pseudo-subriemannian FGLA 
(CPSF), which is called associated with 
the pseudo H-type Lie algebra $(\mathfrak n,\langle\cdot\mid\cdot\rangle)$. 
Given two pseudo H-type Lie algebras $(\mathfrak n,\innerproduct)$ and 
$(\mathfrak n',\innerproduct')$, we say that 
$(\mathfrak n,\innerproduct)$ is isomorphic to  
$(\mathfrak n',\innerproduct')$ if there exists a Lie algebra isomorphism $\varphi$ of 
$\mathfrak n$ onto $\mathfrak n'$ such that $\varphi$ is an isometry of  
$(\mathfrak n,\innerproduct)$ onto $(\mathfrak n',\innerproduct')$. 
Moreover we say that 
$(\mathfrak n,\innerproduct)$ is equivalent to  
$(\mathfrak n',\innerproduct')$ if 
there exists a Lie algebra isomorphism $\varphi$ of 
$\mathfrak n$ onto $\mathfrak n'$ such that: 
(i) $\varphi(\mathfrak n_{-1})=\mathfrak n'_{-1}$, and 
$\varphi|\mathfrak n_{-1}$ is an isometry or an anti-isometry of 
$(\mathfrak n_{-1},\innerproduct_{-1})$ onto $(\mathfrak n'_{-1},\innerproduct'_{-1})$; 
(ii) $\varphi|\mathfrak n_{-2}$ is an isometry of 
$(\mathfrak n_{-2},\innerproduct_{-2})$ onto 
$(\mathfrak n'_{-2},\innerproduct'_{-2})$. 
We denote by $[(\mathfrak n,\innerproduct)]$ the equivalence class containing 
a pseudo H-type algebra $(\mathfrak n,\innerproduct)$. 
If a pseudo H-type Lie algebra $(\mathfrak n,\innerproduct)$ 
is equivalent to a pseudo H-type Lie algebra 
$(\mathfrak n',\innerproduct')$, 
then the prolongation of $(\mathfrak n,[\innerproduct_{-1}])$ is 
isomorphic to that of  $(\mathfrak n',[\innerproduct'_{-1}])$. 
\begin{lemma}
Let $(\mathfrak n=\mathfrak n_{-1}\oplus \mathfrak n_{-2},\innerproduct)$ 
be a pseudo H-type Lie algebra. 
We define a new scalar product $\innerproduct'$ on $\mathfrak n$ as follows: 
$$\inph{x}{y}'=\alpha\inph{x}{y}\ (x,y\in\mathfrak n_{-1}),\quad
\inph{z}{w}'=\beta\inph{z}{w}\ (z,w\in\mathfrak n_{-2}),\quad
\inph{\mathfrak n_{-1}}{\mathfrak n_{-2}}'=0,$$
where $\alpha,\beta$ are nonzero real numbers. 
The pair $(\mathfrak n=\mathfrak n_{-1}\oplus \mathfrak n_{-2},\innerproduct')$ also becomes a pseudo H-type Lie algebra if and only if $\alpha^2=\beta$. 
In this case, 
the CPSF associated with $(\mathfrak n,\innerproduct')$ is 
$(\mathfrak n,[\alpha \innerproduct_{-1}])$. 
\label{lem24}
\end{lemma}
\begin{proof}
By \eqref{eq22}, 
for $x,y\in\mathfrak n_{-1}$ and $z\in\mathfrak n_{-2}$, 
$\inph{\alpha^{-1}\beta J_{z}(x)}{y}'=\beta\inph{J_{z}(x)}{y}
=\beta\inph{z}{[x,y]}=\inph{z}{[x,y]}'$. 
By \eqref{eq23}, 
$(\alpha^{-1}\beta J_{z})^2=\alpha^{-2}\beta^2 J_{z}^2=
-\alpha^{-2}\beta^2\inph{z}{z}1_{\mathfrak n_{-1}}=-\alpha^{-2}\beta\inph{z}{z}'1_{\mathfrak n_{-1}}$. This proves the first statement. 
The last statement is clear. 
\end{proof}
The proof of the following lemma is due to the proof of \cite{FM17:1}*{Theorem 2}. 
\begin{lemma} 
Let $(\mathfrak n^{(1)},\innerproduct^{(1)})$ 
and $(\mathfrak n^{(2)},\innerproduct^{(2)})$ be pseudo H-type Lie algebras. 
Assume that there exists a GLA isomorphism $\varphi$ of $\mathfrak n^{(1)}$ onto $\mathfrak n^{(2)}$. 
Then there exists a GLA isomorphism $\psi$ of $\mathfrak n^{(1)}$ onto $\mathfrak n^{(2)}$ and a positive real number $\alpha$ 
such that: \upshape{(i)} $\psi|\mathfrak n_{-2}^{(1)}$ is an isometry or an anti-isometry; 
\upshape{(ii)} $\psi|\mathfrak n_{-1}^{(1)}=\alpha\varphi|\mathfrak n_{-1}^{(1)}$.  
\label{lem26}
\end{lemma}
\begin{remark} 
\begin{enumerate}
\fixitem
Let $(\mathfrak n^{(1)},\innerproduct^{(1)})$  and $(\mathfrak n^{(2)},\innerproduct^{(2)})$ be H-type Lie algebras. 
If $\mathfrak n^{(1)}$ is isomorphic to $\mathfrak n^{(2)}$ as a GLA, then
$(\mathfrak n^{(1)},\innerproduct^{(1)})$ is isomorphic to 
$(\mathfrak n^{(2)},\innerproduct^{(2)})$ as an H-type Lie algebra
(\cite{KS16:1}*{Theorem 2}). 
\item Let $(\mathfrak n,\innerproduct)$ be a pseudo H-type Lie algebra. 
If $\sgn(\innerproduct_{-2})=(r,s)$, $s>0$, 
then $(\mathfrak n_{-1},\innerproduct_{-1})$ is neutral (\cite{Cia00:1}*{Proposition 2.2}). 
\end{enumerate}
\end{remark}

\begin{proposition}
Let $\gla g$ be a finite dimensional real SGLA 
such that the negative part $\mathfrak g_-=\bigoplus_{p<0}\limits\mathfrak g_p$ is an FGLA of the second kind. 
Let $\innerproduct^{(i)}$ $(i=1,2)$ be scalar products on $\mathfrak g_-$. 
Assume that: 
\begin{enumerate}
\renewcommand{\labelenumi}{\textup{(\roman{enumi})} }
\item $(\mathfrak g_-,\innerproduct^{(1)})$ and 
$(\mathfrak g_-,\innerproduct^{(2)})$ are pseudo H-type Lie algebras 
whose associated FGLAs coincide with $\mathfrak g_-$ as a GLA. 
\item For $i=1,2$ the prolongation of the associated CPSF 
$(\mathfrak g_-,[\innerproduct_{-1}^{(i)}])$ coincides with $\mathfrak g$. 
\end{enumerate}
Then 
\begin{enumerate}
\item $[\innerproduct_{-1}^{(1)}]$ is equal to $[\innerproduct_{-1}^{(2)}]$ or 
$[-\innerproduct_{-1}^{(2)}]$; 
\item $[\innerproduct_{-2}^{(1)}]=[\innerproduct_{-2}^{(2)}]$, 
 
\end{enumerate} 
Consequently, $(\mathfrak g_-,\innerproduct^{(1)})$ is equivalent to 
$(\mathfrak g_-,\innerproduct^{(2)})$.  
\label{prop22}
\end{proposition}
\begin{proof} 
Let $\varphi$ be the identity transformation of $\mathfrak g_-$. 
By the assumption (i) $\varphi$ is a GLA isomorphism of $\mathfrak g_-$ onto itself. 
By Lemma \ref{lem26}, there exists a GLA isomorphism $\psi$ of 
$\mathfrak g_-$ onto itself such that: (a) the restriction $\psi|\mathfrak g_{-2}$ to 
$\mathfrak g_{-2}$ of $\psi$ is an isometry or an anti-isometry; 
(b) there exist a nonzero real number $\alpha'$ 
such that 
$\psi|\mathfrak g_{-2}=\alpha'^2\varphi|\mathfrak g_{-2}$ and 
$\psi|\mathfrak g_{-1}=\alpha'\varphi|\mathfrak g_{-1}$.  
Hence $\alpha'^4\inph{\cdot}{\cdot}^{(2)}_{-2}
=\pm\inph{\cdot}{\cdot}^{(1)}_{-2}$. 
By assumptions (ii) and \cite{Yat18:1}*{Proposition 5.2}, 
$\inph{\cdot}{\cdot}^{(2)}_{-1}$ coincides with 
$\inph{\cdot}{\cdot}^{(1)}_{-1}$ multiplied by a nonzero real number. 
By Lemma \ref{lem24}, 
there exists a nonzero real number $\alpha$ such that 
$\inph{\cdot}{\cdot}^{(2)}_{-1}=\alpha
\inph{\cdot}{\cdot}^{(1)}_{-1}$, 
$\inph{\cdot}{\cdot}^{(2)}_{-2}=\alpha^2
\inph{\cdot}{\cdot}^{(1)}_{-2}$. 
Thus assertions (1) and (2) are proved. 
We define a linear mapping $f$ of $\mathfrak g_-$ into itself as follows: 
$$f(x)=|\alpha|^{-1/2}x\quad (x\in \mathfrak g_{-1}), \qquad
f(z)=|\alpha|^{-1} z \quad (z\in \mathfrak g_{-2});$$ 
then $f$ is a GLA isomorphism and we see that 
$$\begin{aligned}
& \inph{f(x)}{f(y)}^{(2)}=|\alpha|^{-1}\inph{x}{y}^{(2)}=
\sgn(\alpha)\inph{x}{y}^{(1)}\quad (x,y\in\mathfrak g_{-1}), \\
& \inph{f(z)}{f(z')}^{(2)}=|\alpha|^{-2}\inph{z}{z'}^{(2)}=
\inph{z}{z'}^{(1)}\quad (z,z'\in\mathfrak g_{-2}). 
\end{aligned}$$
Hence 
$(\mathfrak g_-,\innerproduct^{(1)})$ is equivalent to 
$(\mathfrak g_-,\innerproduct^{(2)})$. 
\end{proof}
\section{ComH-type Lie algebras}
In this section we introduce comH-type Lie algebras. 
The comH-type Lie algebras consist of comH-type Lie algebras 
$\mathfrak H^{(1)}(\mathbb F,A)$ of the first class, 
$\mathfrak H^{(2)}(\mathbb F,A,\gamma)$ of the second class, 
and $\mathfrak H^{(3)}(\mathbb F,A)$ of the third class, 
which is defined below. 
In particular, if $\mathbb F=\mathbb C,\mathbb C',\mathbb H$ or 
$\mathbb H'$, then the comH-type Lie algebra is said to be 
of classical type. 
\subsection{Cayley algebras}
Let $\mathbb F$ be one of composition algebras $\mathbb C$, $\mathbb C'$, $\mathbb H$, $\mathbb H'$, $\mathbb O$, $\mathbb O'$ over $\mathbb R$. 
We denote by $\mathbb F(\gamma)$ the Cayley extension of $\mathbb F$ defined by $\gamma$, where $\gamma=\pm1$ (cf. \cite{Bou98:1}*{Ch.3, no.5}). Namely $\mathbb F(\gamma)$ is an algebra over 
$\mathbb R$ which $\mathbb F(\gamma)=\mathbb F\times\mathbb F$ as a module and the multiplication on $\mathbb F(\gamma)$ is defined by 
$$(x_1,x_2)(y_1,y_2)
=(x_1y_1+\gamma \overline{y_2}x_2,x_2\overline{y_1}+y_2x_1).$$ 
Clearly $\mathbb F\times\{0\}$ is a subalgebra of $\mathbb F(\gamma)$ isomorphic to $\mathbb F$; we shall identify it with $\mathbb F$ in what follows, 
so that $x\in\mathbb F$ is identified with $(x,0)$. 
Let $\ell=(0,1)$, so that $(x,y)=x+y\ell$ for $x,y\in\mathbb F$. 
Note that: 
(i) $\ell\alpha=\overline{\alpha}\ell$; 
(ii) $\alpha(\beta \ell)=(\beta\alpha)\ell$; 
(iii) $(\alpha \ell)\beta=(\alpha\bar{\beta})\ell$; 
(iv) $(\alpha \ell)(\beta \ell)=\gamma(\overline{\beta}\alpha)$; 
(v) $\ell^2=\gamma$, 
where $\alpha,\beta\in \mathbb F$. 
When $\mathbb F=\mathbb H$ (resp. $\mathbb F=\mathbb H'$) 
we put $\mathbb F_0=\mathbb C$, and 
$\gamma_0=-1$ 
(resp. $\gamma_0=1)$; then 
$\mathbb F=\mathbb F_0(\gamma_0)$. Let $\ell_0$ be the element of $\mathbb F$ corresponding to 
the element $(0,1)\in\mathbb F_0(\gamma_0)=\mathbb F_0\times \mathbb F_0$. 
We denote by $\mathbb F^c=\mathbb F\oplus\sqrt{-1}\mathbb F$, 
$\mathbb F(\gamma)^c=\mathbb F(\gamma)\oplus\sqrt{-1}\mathbb F(\gamma)$ the complexifications 
of $\mathbb F$, $\mathbb F(\gamma)$ respectively. 
Let $\pr_1$ and $\pr_2$ be the projections of 
$\mathbb F(\gamma)^c=\mathbb F^c\times \mathbb F^c$ onto $\mathbb F^c$ 
defined by $\pr_i(x_1,x_2)=x_i$ $(i=1,2)$. 
Note that $\pr_1(\overline{\alpha})=\overline{\pr_1(\alpha)}$, 
$\pr_2(\overline{\alpha})=-\pr_2(\alpha)$, 
$\pr_1(\ell\alpha)=\gamma\overline{\pr_2(\alpha)}$, 
$\pr_2(\ell\alpha)=\overline{\pr_1(\alpha)}$, 
where $\alpha\in \mathbb F(\gamma)^c$. 
We define a mapping $R$ of $\mathbb F(\gamma)^c$ to $\mathbb R$ 
by $R(u+\sqrt{-1}v)=\rea(u)$ $(u,v\in\mathbb F(\gamma))$. 
For $z\in \mathbb F=\mathbb F\times\{0\}$ and $\alpha\in \mathbb F(\gamma)^c$ we obtain $R(z\pr_1(\alpha))=R(z\alpha)$. 
We extend the conjugation ``$\overset{-}{\cdot}$'' on $\mathbb F(\gamma)$ 
to $\mathbb F(\gamma)^c$ by 
$\overline{u+\sqrt{-1}v}=\overline{u}+\sqrt{-1}\overline{v}$. 
\subsection{ComH-type Lie algebras of the first class}
\label{sec32}
Let $\mathbb F$ be $\mathbb C$, $\mathbb C'$, $\mathbb H$, $\mathbb H'$, 
$\mathbb O$ or $\mathbb O'$. Let $A$ be a  real symmetric matrix of order $n$ 
such that $A^2=1_n$ and $\sgn(A)=(r,s)$ $(r\geqq s)$. 

We put 
$$\mathfrak n_{-1}=\mathbb F^n,\quad \mathfrak n_{-2}=\ima \mathbb F, 
\quad \mathfrak n=\mathfrak n_{-1}\oplus \mathfrak n_{-2},$$
where we assume $n=1$ in case $\mathbb F=\mathbb O$ or $\mathbb O'$. 
Note that $\mathbb F^n$ is the set of all the $\mathbb F$-valued row vectors of order $n$. 
We define a bracket operation on $\mathfrak n$ as follows: 
$$[x,y]=-2\ima(xSy^*)=yAx^*-xAy^* \quad (x,y\in \mathfrak n_{-1}), 
\quad [\mathfrak n_{-1},\mathfrak n_{-2}]=[\mathfrak n_{-2},\mathfrak n_{-2}]=0;$$ 
then $(\mathfrak n,[\cdot,\cdot])$ becomes an FGLA of the second kind. 
Furthermore we define a symmetric bilinear form $\langle \cdot\mid \cdot\rangle$ on $\mathfrak n$ 
as follows: 
$$\begin{aligned}
&
\langle x\mid y\rangle=2\rea(xAy^*)\quad (x,y\in\mathfrak n_{-1}), \\
& \langle z\mid w\rangle=\rea(z\overline{w})=-\rea(zw)\quad (z,w\in\mathfrak n_{-2}),
\quad \langle \mathfrak n_{-1}\mid \mathfrak n_{-2}\rangle=0. 
\end{aligned}$$
The linear mapping $J_z$ defined by \eqref{eq23} has the following form: 
$J_z(x)=-zx$. 
Thus $(\mathfrak n,\langle\cdot\mid\cdot\rangle)$ becomes a pseudo H-type Lie algebra, which is denoted by 
$\mathfrak H^{(1)}(\mathbb F,A)
=(\mathfrak h^{(1)}(\mathbb F,A),\innerproduct)$. 
The comH-type Lie algebra $\mathfrak H^{(1)}(\mathbb F,A)$ 
is called a comH-type Lie algebra of the first class. 
We denote the FGLA associated with $\mathfrak H^{(1)}(\mathbb F,A)$ 
by $\mathfrak h^{(1)}(\mathbb F,A)=\bigoplus\limits_{p=-1}^{-2}
\mathfrak h^{(1)}(\mathbb F,A)_p$. 

\begin{lemma}
Let $(r,s)$ be the signature of $A$. Let $\mathbb F$ be $\mathbb C$, $\mathbb C'$, $\mathbb H$, $\mathbb H'$, 
$\mathbb O$ or $\mathbb O'$. 
\begin{enumerate}
\item $\mathfrak H^{(1)}(\mathbb F,A)$ is  
isomorphic to 
$\mathfrak H^{(1)}(\mathbb F,1_{r,s})$. 
\item  $\mathfrak H^{(1)}(\mathbb F,1_{r,s})$ is equivalent to 
$\mathfrak H^{(1)}(\mathbb F,1_{s,r})$. 
\end{enumerate}
\label{lem31}
\end{lemma}
\begin{proof}
(1) 
There exists a real orthogonal matrix $P$ such that 
$PAP^{-1}=1_{r,s}$. 
We define a linear mapping $\varphi$ of 
$\mathfrak h^{(1)}(\mathbb F,1_{r,s})$ to 
$\mathfrak h^{(1)}(\mathbb F,A)$ as follows: 
$$\varphi(x)=
xP\quad (x\in \mathfrak h^{(1)}(\mathbb F,1_{r,s})_{-1}), \quad 
\varphi(z)=z\quad (z\in \mathfrak h^{(1)}(\mathbb F,1_{r,s})_{-2}).$$  
Then $\varphi$ is an isomorphism as a pseudo H-type Lie algebra. 
Hence  $\mathfrak H^{(1)}(\mathbb F,A)$ is  
isomorphic to 
$\mathfrak H^{(1)}(\mathbb F,1_{r,s})$. 

(2) We define a linear mapping $\psi$ of 
$\mathfrak h^{(1)}(\mathbb F,1_{r,s})$ to 
$\mathfrak h^{(1)}(\mathbb F,1_{s,r})$ as follows: 
$$\psi(x)=
xK_n \ (x\in \mathfrak h^{(1)}(\mathbb F,1_{r,s})_{-1}), \quad 
\psi(z)=-z\ (z\in \mathfrak h^{(1)}(\mathbb F,1_{r,s})_{-2})$$ 
Then $\psi$ is an isomorphism 
as a GLA. 
Moreover $\psi|\mathfrak h^{(1)}(\mathbb F,1_{r,s})_{-2}$ is isometry and 
$\psi|\mathfrak h^{(1)}(\mathbb F,1_{r,s})_{-1}$ is anti-isometry. 
Hence $\mathfrak H^{(1)}(\mathbb F,1_{r,s})$ is equivalent to 
$\mathfrak H^{(1)}(\mathbb F,1_{s,r})$. 
\end{proof}

\begin{remark}
The pseudo H-type Lie algebra 
$\mathfrak H^{(1)}(\mathbb F,1_{r,s})$ is isomorphic to 
$\mathfrak h'_{r,s}(\mathbb F)$ in 
\cite{KS17:1} as a GLA. 
\label{rem31}
\end{remark}

\subsection{ComH-type Lie algebras of the second and the third classes}
\label{sec31}
Let $\mathbb F$ be $\mathbb C$, $\mathbb C'$, $\mathbb H$, $\mathbb H'$, 
$\mathbb O$ or $\mathbb O'$. 
We set 
$$\mathfrak g_{-1}=(\mathbb F(\gamma)^c)^n, \qquad 
\mathfrak g_{-2}=\mathbb F^c,$$ 
where we assume $n=1$ in case $\mathbb F=\mathbb O$ or $\mathbb O'$. 
Let $A$ be a real symmetric matrix of order $n$ such that $A^2=1_n$ and $\sgn(A)=(r,s)$ $(r\geqq s)$. 
We define a bracket operation $[\cdot\,,\,\cdot]$ on 
$\mathfrak m=\mathfrak g_{-2}\oplus \mathfrak g_{-1}$ as follows: 
$$[\alpha,\beta]=\pr_2(\alpha A\beta^*)\quad 
(\alpha,\beta\in \mathfrak g_{-1}),\quad 
[\mathfrak g_{-1},\mathfrak g_{-2}]=[\mathfrak g_{-2},\mathfrak g_{-2}]=0.$$ 
More explicitly, the bracket operation can be written as follows: 
if we put $\alpha=\alpha_1+\alpha_2\ell$ and $\beta=\beta_1+\beta_2\ell$ 
$(\alpha_1,\alpha_2,\beta_1,\beta_2\in(\mathbb F^c)^n$), then 
$$[\alpha,\beta]=\alpha_2A\transpose{\beta}_1-\beta_2A\transpose{\alpha}_1.$$ 
Then $\mathfrak m$ becomes a complex FGLA of the second kind. 
Moreover we define a symmetric bilinear form $\innerproduct$ on 
$\mathfrak m$ as follows:
$$\begin{aligned}
& \inph{\alpha}{\beta}=
R(\alpha A\beta^*)\quad   
(\alpha,\  \beta\in \mathfrak g_{-1}), \\
& \inph{z_1}{z_2}=-\gamma R(\overline{z_1}{z_2})\quad (z_1,z_2\in\mathfrak g_{-2}),\quad \inph{\mathfrak g_{-1}}{\mathfrak g_{-2}}=0. 
\end{aligned}$$
More explicitly, the bilinear form can be written as follows: 
if we put 
$\alpha=\alpha_1+\alpha_2\ell$ and $\beta=\beta_1+\beta_2\ell$ 
$(\alpha_1,\alpha_2,\beta_1,\beta_2\in(\mathbb F^c)^n$), 
then 
$$\inph{\alpha}{\beta}=R(\alpha_1A{}^t\overline{\beta_1}-
\gamma\overline{\beta_2}A\transpose{\alpha}_2).$$
For $z\in\mathfrak g_{-2}$ 
the linear mapping $J_z$ of $\mathfrak g_{-1}$ to itself defined by  
$$\inph{J_z(x)}{y}=\inph{z}{[x,y]}\qquad (x,y\in\mathfrak g_{-1})$$ 
satisfies 
$$J_z(\alpha)=-(z\ell)\alpha,\qquad 
J_z^2=\gamma\overline{z}z1_{\mathfrak g_{-1}}.$$ 
We denote by the same letter $\tau$ the conjugations of $\mathbb F^c$ and $\mathbb F(\gamma)^c$  
with respect to $\mathbb F$ and $\mathbb F(\gamma)$ respectively. 
We now extend $\tau$ to a grade-preserving involution of $\mathfrak m$ 
in a natural way, which is also denoted by the same letter. 
Next we define a grade-preserving involution $\kappa$ of $\mathfrak m$ as follows: 
$$\kappa(\alpha)=-\overline{\alpha_2}-\overline{\alpha_1}\ell,\qquad 
\kappa(z)=-\overline{z}, $$
where $\alpha=\alpha_1+\alpha_2\ell\in\mathfrak g_{-1}$ 
$(\alpha_1,\alpha_2\in(\mathbb F^c)^n$, 
$z\in\mathfrak g_{-2}$). 
We denote by $\mathfrak n^1$ and $\mathfrak n^2$ the sets of elements which are fixed under $\tau$ and  $\kappa\comp\tau$ respectively. 
Then $\mathfrak n^1$ and $\mathfrak n^2$ become graded subalgebras of 
$\mathfrak m_\mathbb R$ with 
$$\mathfrak n^i=\bigoplus_{p<0}\mathfrak n^i_p,\qquad 
\mathfrak n^i_p=\mathfrak n^i\cap\mathfrak g_p.$$
Explicitly the subspaces $\mathfrak n^i_p$ are described as follows: 
$$\begin{aligned}
&\mathfrak n^1_{-1}=\mathbb F(\gamma)^n, \qquad 
\mathfrak n^1_{-2}=\mathbb F, \\
& \mathfrak n^2_{-1}=\{\alpha_1+\hat{\tau}(\alpha_1)\ell:
\alpha_1\in(\mathbb F^c)^n\}, \qquad 
\mathfrak n^2_{-2}=\sqrt{-1}\mathbb R\oplus\ima(\mathbb F),
\end{aligned}$$
where $\hat{\tau}$ is a mapping of $\mathbb F^c$ to itself defined 
by $\hat{\tau}(x)=-\tau(\overline{x})$. 
We note that the bracket operation and the scalar product on $\mathfrak n^2$ can be written as follows: 
if we put $\alpha=\alpha_1+\hat{\tau}(\alpha_1)\ell$ and 
$\beta=\beta_1+\hat{\tau}(\beta_1)\ell$ 
$(\alpha_1,\beta_1\in(\mathbb F^c)^n$), then 
$$\begin{aligned}
&[\alpha,\beta]=\hat{\tau}(\alpha_1)A\transpose{\beta}_1
-\hat{\tau}(\beta_1)A\transpose{\alpha}_1, \\ 
& \langle\alpha\mid\beta\rangle=
R(\alpha_1A\transpose{\overline{\beta_1}}-\gamma\tau(\beta_1)A
\transpose{\tau}(\overline{\alpha_1}))
=(1-\gamma)R(\alpha_1A\transpose{\overline{\beta_1}}).
\end{aligned}$$ 
We always assume that $\gamma=-1$ when we consider the case 
$\mathfrak n^2$. 
Since $z\overline{z}\in\mathbb R$ for $z\in\mathfrak n^i_{-2}$ $(i=1,2)$, 
$\mathfrak n^1$ and $\mathfrak n^2$ are pseudo H-type Lie algebras. 

\begin{enumerate}
\item Pseudo H-type Lie algebra $\mathfrak n^1$. 
We assume that $\gamma=-1$ if $\mathbb F=\mathbb C',\mathbb H',\mathbb O'$. 
The pseudo H-type Lie algebra $(\mathfrak n^1,\innerproduct)$ 
is called a comH-type Lie algebra of the second class, 
which is denoted by $\mathfrak H^{(2)}(\mathbb F,A,\gamma)
=(\mathfrak h^{(2)}(\mathbb F,A,\gamma),\innerproduct)$. 
We denote the FGLA associated with 
$\mathfrak H^{(2)}(\mathbb F,A,\gamma)$ 
by $\mathfrak h^{(2)}(\mathbb F,A,\gamma)=\bigoplus\limits_{p=-1}^{-2}
\mathfrak h^{(2)}(\mathbb F,A,\gamma)_p$ . 
Note that $\mathfrak h^{(2)}(\mathbb C,A,\gamma)$ becomes 
a complex FGLA.  

\item  Pseudo H-type Lie algebra $\mathfrak n^2$. 
We assume that $\mathbb F=\mathbb H$, $\mathbb H'$, $\mathbb O$ or 
$\mathbb O'$.  
The pseudo H-type Lie algebra $(\mathfrak n^2,\innerproduct)$ 
is called a comH-type Lie algebra of the third class, 
which is denoted by $\mathfrak H^{(3)}(\mathbb F,A)
=(\mathfrak h^{(3)}(\mathbb F,A),\innerproduct)$. 
We denote the FGLA associated with $\mathfrak H^{(3)}(\mathbb F,A)$ 
by 
$\mathfrak h^{(3)}(\mathbb F,A)=\bigoplus\limits_{p=-1}^{-2}
\mathfrak h^{(3)}(\mathbb F,A)_p$. 
\end{enumerate}
\begin{lemma} Let $\mathbb F$ be $\mathbb H$, $\mathbb H'$, $\mathbb O$ or $\mathbb O'$. 
Let $(r,s)$ be the signature of $A$ and $\gamma,\gamma'\in\{\pm1\}$. 
\begin{enumerate}
\item 
$\mathfrak H^{(2)}(\mathbb F,A,\gamma)$ (resp. 
$\mathfrak H^{(3)}(\mathbb F,A)$)  
is isomorphic to 
$\mathfrak H^{(2)}(\mathbb F,1_{r,s},\gamma)$ 
(resp. $\mathfrak H^{(3)}(\mathbb F,1_{r,s})$). 
\item    
$\mathfrak h^{(2)}(\mathbb F,A,\gamma')$ is  
isomorphic to 
$\mathfrak h^{(2)}(\mathbb F,1_{r+s},\gamma)$ 
as a GLA. 
\item $\mathfrak H^{(2)}(\mathbb F,1_{r,s},\gamma)$  
(resp. $\mathfrak H^{(3)}(\mathbb F,1_{r,s})$)
is equivalent to 
$\mathfrak H^{(2)}(\mathbb F,1_{s,r},\gamma)$ 
(resp. $\mathfrak H^{(3)}(\mathbb F,1_{s,r})$).   
\item 
$\mathfrak H^{(3)}(\mathbb H',1_{r,s})$ is isomorphic to 
$\mathfrak H^{(3)}(\mathbb H',1_{r+s})$. 
\end{enumerate}
\label{lem32}
\end{lemma}
\begin{proof}
As in Lemma \ref{lem31} we can prove (1) and (3).  

(2) There exists a real orthogonal matrix $P$ such that 
$PAP^{-1}=1_{r,s}$. We define a linear mapping of 
$\mathfrak h^{(2)}(\mathbb F,1_{r+s},\gamma')$ to 
$\mathfrak h^{(2)}(\mathbb F,A,\gamma)$ as follows: 
$$\varphi(\alpha_1+\alpha_2\ell)=
\alpha_1P+\alpha_21_{r,s}P\ell\quad(\alpha_1,\alpha_2\in \mathbb F^n), 
\qquad \varphi(z)=z
\quad(z\in\mathfrak h^{(2)}(\mathbb F,1_{r+s},\gamma')_{-2}).$$ 
Then $\varphi$ is an isomorphism 
as a GLA. 

(4) Let $(1,i',j',k')$ be the canonical basis of $\mathbb H'$. 
We define 
a linear mapping of 
$\mathfrak h^{(3)}(\mathbb H',1_{r+s})$ to 
$\mathfrak h^{(3)}(\mathbb H',1_{r,s})$ as follows: 
$$\begin{aligned}
\varphi(\alpha)& =\varphi_1(\alpha)+\hat{\tau}(\varphi_1(\alpha))\ell
\quad (\alpha\in\mathfrak h^{(3)}(\mathbb H',1_{r+s})_{-1}), \\
\varphi_1(\alpha)&=(\alpha_r,j'\alpha_s), \quad 
(\alpha=\alpha_1+\hat{\tau}(\alpha_1)\ell, \alpha_1=(\alpha_r,\alpha_s), 
\alpha_1\in (\mathbb H^{\prime c})^n, 
\alpha_r\in(\mathbb H^{\prime c})^r, \alpha_s\in(\mathbb H'^c)^s) \\
\varphi(z)&=z\quad (z\in\mathfrak h^{(3)}(\mathbb H',1_{r+s})_{-2}).
\end{aligned}$$ 
Then $\varphi$ is an isomorphism of 
$\mathfrak H^{(3)}(\mathbb H',1_{r+s})$ onto 
$\mathfrak H^{(3)}(\mathbb H',1_{r,s})$. 
\end{proof}
\begin{remark}
The comH-type Lie algebra 
$\mathfrak H^{(2)}(\mathbb F,1_{r+s},-1)$ isomorphic to 
$\mathfrak h_{r+s}(\mathbb F)$ in 
\cite{KS16:1} as a GLA. 
\label{rem32}
\end{remark}
\begin{remark} 
\begin{enumerate}
\fixitem  
Let $(\mathfrak n,\innerproduct)$ be a comH-type Lie algebra of 
the first class with $\sgn(A)=(r,s)$. We put $\dim_\mathbb R\mathbb F=a$. 
\begin{enumerate}
\renewcommand{\labelenumii}{(\alph{enumii}) }
\item 
If $\mathbb F=\mathbb C,\mathbb H,\mathbb O$, then 
$\sgn(\innerproduct_{-1})=(ra,sa)$ and 
$\sgn(\innerproduct_{-2})
=(a-1,0)$. 
\item 
If $\mathbb F=\mathbb C',\mathbb H',\mathbb O'$, then 
$\sgn(\innerproduct_{-1})
=((r+s)a/2,(r+s)a/2)$ and 
$\sgn(\innerproduct_{-2})
=(a/2-1,a/2)$. 
\end{enumerate}
\item 
Let $(\mathfrak n,\innerproduct)$ be a comH-type Lie algebra of 
the second class with $\sgn(A)=(r,s)$. 
\begin{enumerate}
\renewcommand{\labelenumii}{(\alph{enumii}) }
\item 
If $\mathbb F=\mathbb C,\mathbb H,\mathbb O$ and $\gamma=-1$, then 
$\sgn(\innerproduct_{-1})
=(2ra,2sa)$ 
and $\innerproduct_{-2})=(a,0)$. 
\item 
If $\mathbb F=\mathbb C,\mathbb H,\mathbb O$ and $\gamma=1$, then 
$\sgn(\innerproduct\rangle_{-1})
=((r+s)a,(r+s)a)$ 
and  $\sgn(\innerproduct\rangle_{-2})
=(0,a)$. 
\item If $\mathbb F=\mathbb C',\mathbb H',\mathbb O'$ and $\gamma=-1$, 
then 
$\sgn(\innerproduct_{-1})
=((r+s)a,(r+s)a)$ 
and $\sgn(\innerproduct\rangle_{-2})
=(a/2,a/2)$. 
\end{enumerate}
\item 
If $(\mathfrak n,\innerproduct)$ be a comH-type Lie algebra of 
the third class with $\sgn(A)=(r,s)$. 
\begin{enumerate}
\renewcommand{\labelenumii}{(\alph{enumii}) }
\item If $\mathbb F=\mathbb H$ or $\mathbb O$, then 
$\sgn(\innerproduct_{-1})
=((r+s)a,(r+s)a)$ and $\sgn(\innerproduct_{-2})=(a-1,1)$. 
\item If $\mathbb F=\mathbb H'$ or $\mathbb O'$, then 
$\sgn(\innerproduct_{-1})
=((r+s)a,(r+s)a)$ and $\sgn(\innerproduct_{-2})=(a/2-1,a/2+1)$. 
\end{enumerate}
\end{enumerate}
\end{remark}

\section{Prolongations of FGLAs associated with comH-type Lie algebras}
In this section we first matricial representations of comH-type Lie algebras of classical type and determine the prolongation. 
\subsection{General results}
Let $(\mathfrak n,\innerproduct)$ be a pseudo H-type Lie algebra, and 
let $\pcgla[n]{g}$ be the prolongation of $\mathfrak n$. 
The natural inclusion $\iota$ of 
$\mathfrak{so}(\mathfrak n_{-2},\innerproduct_{-2})$ into 
$\mathfrak g(\mathfrak n)_0$ is defined by 
$$[\iota(v\wedge u),x]=\frac14[J_v,J_u](x)\ 
(x\in\mathfrak n_{-1}),\quad 
[\iota(v\wedge u),z]=(v\wedge u)(z)\ 
(z\in\mathfrak n_{-2}),$$
where $v\wedge u$ is the skew-symmetric endomorphism 
$\inph{v}{\cdot}u-\inph{u}{\cdot}v$. 

Here we quote useful results from \cite{AS14:1} and \cite{AS14:2}. 
\begin{proposition}[{\cite{AS14:2}*{Theorem 2.3}}]
Let $(\mathfrak n,\innerproduct)$ be a pseudo H-type Lie algebra, 
and let $\pcgla[n]{g}$ be the prolongation of $\mathfrak n$. Then 
$$\mathfrak g(\mathfrak n)_0=\mathfrak{so}(\mathfrak n_{-2},\innerproduct_{-2})\oplus\mathbb RE\oplus \check{\mathfrak h}_0,$$
where $E$ is the characteristic element of the GLA $\pcgla[n]g$ and 
$\check{\mathfrak h}_0=\{~x\in\mathfrak g(\mathfrak n)_0:
[x,\mathfrak n_{-2}]=0~\}$.  
\label{prop41}
\end{proposition}

\begin{theorem}[{\cite{AS14:2}*{Theorem 3.1 and Remark 3.2}}]
Let $(\mathfrak n,\innerproduct)$ be a pseudo H-type Lie algebra with 
$\dim\mathfrak n_{-2}\geqq3$, and let $\pcgla[n]{g}$ be the prolongation of $\mathfrak n$. 
If $\mathfrak g(\mathfrak n)_1\ne0$, 
then $\pcgla[n]{g}$ is a finite dimensional SGLA. 
\label{th41}
\end{theorem}

Let $(\mathfrak n,\innerproduct)$ be a pseudo H-type Lie algebra with 
$\dim\mathfrak n_{-2}\geqq3$. 
Since a pseudo H-type Lie algebra is a real extended translation algebra, 
if the prolongation of $\mathfrak n$ is simple, 
then $\dim \mathfrak n_{-2}=3,4,7$ or $8$ 
(\cite{AS14:1}*{Theorem 3.6}). 
Hence by Theorem \ref{th41} we obtain the following

\begin{corollary}  \label{cor41}
Let $(\mathfrak n,\innerproduct)$ be a pseudo $H$-type Lie algebra 
with $\dim\mathfrak n_{-2}\geqq3$ and $\pcgla[n]g$ be the prolongation of 
$\mathfrak n$. 
If $\dim \mathfrak n_{-2}\ne 3,4,7,8$, then 
$\mathfrak g(\mathfrak n)_p=0$ for all $p\geqq1$. 
\end{corollary}
\subsection{Pseudo H-type Lie algebras with $\dim\mathfrak n_{-2}\leqq2$} 
\label{sec411}
Let $(\mathfrak n,\innerproduct)$ be a pseudo H-type Lie algebra. 
\begin{enumerate} 
\item The case $\dim\mathfrak n_{-2}=1$. 
The prolongation $\pcgla[n]g$ of $\mathfrak n$ is isomorphic to a real  contact algebra 
$K(N/2,\mathbb R)$, 
where $N=\dim\mathfrak n_{-1}$ (For the details of contact algebras, see \cite{Kac68:1}). 
\item The case $\dim\mathfrak n_{-2}=2$. 
Let $(z_1,z_2)$ be a basis of $(\mathfrak n_{-2},\innerproduct_{-2})$ such that 
$\inph{z_i}{z_j}=\varepsilon_i\delta_{ij}$, where $\varepsilon_i=\pm1$.  
We define an endomorphism $I$ of 
$\mathfrak n$ as follows: 
$$
I(x)=J_{z_1}J_{z_2}(x),\quad  I(z)= z_1\wedge z_2(z),$$
then $I$ satisfies 
$$I^2=-\varepsilon_1\varepsilon_21_{\mathfrak n}, 
\quad [Ix,y]=I[x,y], \quad \inph{Ix}{y}+\inph{x}{Iy}=0.$$ 
\begin{enumerate}
\renewcommand{\labelenumii}{(2\alph{enumii}) }
\item The case $\sgn(\innerproduct_{-2})=(2,0)$ or $(0,2)$. 

In this case $\varepsilon_1\varepsilon_2=1$ and hence $I$ is a complex structure of $\mathfrak n$. The complex structure on $\mathfrak n$ is naturally extended that on 
the prolongation $\pcgla[n]g$ of $\mathfrak n$, which is denoted by the same letter. 
Furthermore $\pcgla[n]g$ is isomorphic to the realization of a complex contact algebra $K(N/4,\mathbb C)$, 
where $N=\dim \mathfrak n_{-1}$.   
\item The case $\sgn(\innerproduct_{-2})=(1,1)$. 

In this case $\varepsilon_1\varepsilon_2=-1$ and hence $I$ is a paracomplex structure of $\mathfrak n$. The paracomplex structure $I$ on $\mathfrak n$ is naturally extended that on 
the prolongation $\pcgla[n]g$ of $\mathfrak n$, which is denoted by the same letter. 
We set 
$$\mathfrak g(\mathfrak n)^\pm_p=
\{X\in\mathfrak g(\mathfrak n)_p:[I,X]=\pm X\},
\quad 
\mathfrak g(\mathfrak n)^\pm=\bigoplus\limits_{p\in\mathbb Z}\mathfrak g(\mathfrak n)_p^\pm$$
Then the prolongation $\pcgla[n]g$ of 
the FGLA $\mathfrak n$ is 
the direct sum of $\mathfrak g(\mathfrak n)^+$ and $\mathfrak g(\mathfrak n)^-$. 
Furthermore 
$\mathfrak g(\mathfrak n)^\pm=\bigoplus_{p\in\mathbb Z}\limits\mathfrak g(\mathfrak n)^\pm_p$ are isomorphic to a contact algebra $K(N/4,\mathbb R)$. 
\end{enumerate}
\end{enumerate}
\subsection{Matricial models of comH-type Lie algebras of the first class}
Let $\mathbb F$ be $\mathbb C$, $\mathbb H$, 
$\mathbb C'$ or $\mathbb H'$. 
We put 
$\mathfrak l=
\mathfrak{sl}(n+2,\mathbb F)$ $(n\geqq1)$; 
then $\mathfrak l$ is a real semisimple Lie algebra. 

We put 
$
\mathfrak{s}=\{\,X\in\mathfrak{sl}(n+2,\mathbb F):
X^*S_{p,q}+S_{p,q}X=O\,\}$ $(n\geqq1,2p+q=n+2,p\geqq 1,q\geqq0)$; then 
$$\mathfrak{s}=\left\{
X=\begin{bmatrix}
X_{11} & X_{12} & X_{13} \\
X_{21} & X_{22} & -S_{p-1,q}X^*_{12} \\
X_{31} & -X^*_{21}S_{p-1,q} & -\overline{X_{11}} \\
\end{bmatrix}
:
\begin{aligned} 
& X_{11}\in \mathbb F,\ X_{12}\in M(1,n,\mathbb F), \\ 
& X_{21}\in M(n,1,\mathbb F), \\ 
&X_{31},X_{13}\in \ima \mathbb F, 
X_{22}\in \mathfrak{gl}(n,\mathbb F), \\
&X_{22}+S_{p-1,q}X^*_{22}S_{p-1,q}=O.
\end{aligned}\right\},
$$ 
Here $M(p,q,\mathbb F)$ denotes the set of $\mathbb F$-valued $p\times q$-matrices. 
We define subspaces $\mathfrak s_p$ of $\mathfrak s$ as follows: 
$$\begin{aligned}
&\mathfrak s_{-2}
 =\left\{~0_2\ominus 0_n \ominus X_{31}\in\mathfrak s:
X_{31}\in\ima \mathbb F
\right\}, \\
& 
\mathfrak s_{-1}=\left\{\begin{bmatrix}
0 & 0 & 0 \\
X_{21} & 0 & 0 \\
0 & -X^*_{21}S_{p-1,q} & 0 \\
\end{bmatrix}\in\mathfrak s: 
X_{21}\in M(n,1,\mathbb F)\right\}, \\
& 
\mathfrak s_{0}=\left\{
X_{11} \oplus X_{22}\oplus (-\overline{X_{11}}) 
\in\mathfrak s: 
\begin{aligned} 
& X_{11}\in \mathbb F, 
X_{22}\in \mathfrak{gl}(n,\mathbb F), \\
& X_{22}+S_{p-1,q}X^*_{22}S_{p-1,q}=O
\end{aligned}
\right\}, \\
& 
\mathfrak s_p=\{~X\in\mathfrak s:
{}^tX\in\mathfrak s_{-p}~\}\quad (p=1,2), 
\quad \mathfrak s_p=\{0\} \quad (|p|>2). 
\\
\end{aligned}$$
Then $\gla s$ becomes a GLA whose negative part $\mathfrak s_-$ is an FGLA of the second kind. 
We define a linear mapping of $\mathfrak h^{(1)}(\mathbb F,S_{p-1,q})$ into 
$\mathfrak s_-$ as follows: 
$$\varphi(x)=
\begin{bmatrix}
0 & 0 & 0 \\
x & 0 & 0 \\
0 & -x^*S_{p-1,q} & 0 \\
\end{bmatrix}
\quad (x\in \mathbb F^{p+q-1}),\quad \varphi(z)=\begin{bmatrix}
0 & 0 & 0 \\
0 & 0 & 0 \\
z & 0 & 0 \\
\end{bmatrix}\quad (z\in\mathfrak n_{-2});
$$
then $\varphi$ becomes a GLA isomorphism. 
We define a symmetric bilinear form $\innerproduct $ on $\mathfrak s_{-}$ as follows:  
$$\begin{aligned}
\inph{X}{Y}&=2\rea\tr(XS_{p,q}Y^*) \quad (X,Y\in\mathfrak s_{-1}),\quad 
\inph{X}{Y}=\rea\tr(XY^*) \quad (X,Y\in\mathfrak s_{-2}),\\ 
\inph{X}{Y}&=0\quad (X\in\mathfrak s_{-2},Y\in\mathfrak s_{-1})
\end{aligned}$$
Then $(\mathfrak s_{-},\innerproduct)$ becomes a pseudo H-type Lie algebra 
and $\varphi$ is isomorphism of $\mathfrak H^{(1)}(\mathbb F,S_{p-1,q})$ 
onto $(\mathfrak s_{-},\innerproduct)$. 
Since 
$\ad(\mathfrak s_0)|\mathfrak s_{-1}\subset\mathfrak{co}(\mathfrak s_{-1},\innerproduct_{-1})$, the CPSF $(\mathfrak s_-,[\innerproduct_{-1}])$ is of semisimple type. 

Let $\mathcal H^{(1)}$ be the set of all equivalence classes of comH-type Lie algebras of the first class. 
Let $\mathcal P^{(1)}$ be the set of all isomorphism classes (as a GLA) of real SGLAs 
$\gla s$ given in Table \ref{table1}. 
We define a mapping $\varPhi_1$ of $\mathcal H^{(1)}$ into $\mathcal P^{(1)}$ as follows: 
for $[(\mathfrak n,\innerproduct)]\in \mathcal H^{(1)}$ we define 
$\varPhi_1([(\mathfrak n,\innerproduct)])$ as the equivalence class 
of the prolongation of $(\mathfrak n,[\innerproduct_{-1}])$. 
More precisely $\varPhi_1$ is defined according to Table \ref{table1}.  
From the above results, \cite{Gom96:1}*{\S3} and \cite{Yat18:1}, $\varPhi_1$ is well-defined and surjective.  By Lemmas \ref{lem24} and \ref{lem31}, $\Phi_1$ is injective. 
Thus we obtain the following proposition. 
\begin{proposition} \label{prop42}
$\varPhi_1$ is bijective. 
\end{proposition} 
\begin{table}[hbtp]
\caption{First class}
\label{table1}
\begin{tabular}{|c|c|c|c|p{8cm}|}
\hline
$\mathbb F$ & $\sgn(\innerproduct_{-2})$ & $A$ & $\mathfrak s$ & 
the gradation of $\mathfrak s$ \\ \hline\hline 
$\mathbb C$ & $(1,0)$ & $S_{p-1,q}$ & $\mathfrak{su}(p+q,p) $ & 
$(\textrm{(AIIIa)}_{\ell,p},\{\alpha_1,\alpha_\ell\})$ 
{\tiny $(\ell=2p+q-1, p\geqq2,q\geqq1)$} \\ \hline
$\mathbb C$ & $(1,0)$ & $S_{p-1,0}$ & $\mathfrak{su}(p,p) $ & 
$(\textrm{(AIIIb)}_{\ell},\{\alpha_1,\alpha_\ell\})$ 
{\tiny $(\ell=2p-1, p\geqq2)$} \\ \hline
$\mathbb C$ & $(1,0)$ & $S_{0,q}$ & $\mathfrak{su}(1+q,1) $ & 
$(\textrm{(AIV)}_{\ell},\{\alpha_1,\alpha_\ell\})$ 
{\tiny $(\ell=q+1,q\geqq1)$} 
 \\ \hline
$\mathbb C'$ & $(0,1)$ & $S_{p-1,q}$ & $\mathfrak{sl}(2p+q,\mathbb R)$ & $(\textrm{(AI)}_\ell,\{\alpha_1,\alpha_\ell\})$ \\ \hline
$\mathbb H$ &  $(3,0)$ & $S_{p-1,q}$ & $\mathfrak{sp}(p+q,p)$ & $(\textrm{(CIIa)}_{\ell,p},\{\alpha_2\})$ 
{\tiny $(\ell=2p+q\geqq3,p,q\geqq1)$}, \\ \hline 
$\mathbb H$ &  $(3,0)$ & $S_{p-1,0}$ & $\mathfrak{sp}(p,p)$ &
$(\textrm{(CIIb)}_{\ell},\{\alpha_2\})$ 
{\tiny ($\ell=2p\geqq3$)} 
  \\ \hline
$\mathbb H'$ & $(1,2)$ & $S_{p-1,q}$ & $\mathfrak{sp}(2p+q,\mathbb R)$& $(\textrm{(CI)}_\ell,\{\alpha_2\})$ 
{\tiny $(\ell=2p+q\geqq3)$} \\ \hline
$\mathbb O$ & $(7,0)$ & $1$ & $\mathrm{FII}$ & $(\mathrm{FII},\{\alpha_4\})$ \\ \hline
$\mathbb O'$ &  $(3,4)$ & $1$ &$\mathrm{FI}$ & $(\mathrm{FI},\{\alpha_4\})$ \\ \hline
\end{tabular}
\end{table}

\subsection{Matricial Models of comH-type Lie algebras of the second class} 
Let $\mathbb F=\mathbb C, \mathbb C',\mathbb H,\mathbb H'$. 
Let $\gla s$ be a finite dimensional semisimple GLA $\mathfrak{sl}(n+2,\mathbb F)$ with the the following gradation $(\mathfrak s_p)_{p\in\mathbb Z}$. 

$$\begin{aligned}
&\mathfrak s_{-2}
 =\{~0\ominus 0_n\ominus X_{31}\in\mathfrak s:
X_{31}\in \mathbb F~\}, \\
&\mathfrak s_{-1}
 =\left\{\begin{bmatrix}
0 & 0 & 0   \\
X_{21} & 0 & 0 \\
0& X_{32} & 0  \\
\end{bmatrix}\in\mathfrak s:
X_{21}\in M(n,1,\mathbb F), X_{32}\in M(1,n,\mathbb F)
\right\}, \\
\end{aligned}$$
Note that $\gla s$ is an SGLA except for the case $\mathbb F=\mathbb C'$. 
We consider an FGLA $\mathfrak H^{(2)}(\mathbb F,A,\gamma)$.
That is, 
$$\mathfrak h^{(2)}(\mathbb F,A,\gamma)_{-1}=\mathbb F(\gamma)^n,\quad
\mathfrak h^{(2)}(\mathbb F,A,\gamma)_{-2}=\mathbb F,$$
where $A$ is a real symmetric matrix of order $n$ such that $A^2=1_n$. 
We define a linear mapping $\varphi$ of $\mathfrak h^{(2)}(\mathbb F,A,\gamma)$ to $\mathfrak s_-$ 
as follows:
$$\varphi(\alpha_1+\alpha_2\ell)=
\begin{bmatrix}
0 & 0 & 0  \\
{}^t\alpha_1 & 0 & 0   \\
0& \alpha_2A & 0  \\
\end{bmatrix},\qquad 
\varphi(z)=0_1\ominus 0_n\ominus z.$$
Then $\varphi$ is a GLA isomorphism. 
Moreover we define a nondegenerate symmetric bilinear form 
on $\mathfrak g_-$ as follows: 
$$\begin{aligned}
& \inph{X}{Y}=\rea(\transpose{x_{21}}A\overline{y_{21}}
-\gamma x_{32}Ay_{32}^*)
, \\
& \inph{Z}{W}=-\gamma\rea(z_{31}\overline{w_{31}})
\quad (Z,W\in\mathfrak s_{-2}),\quad \inph{\mathfrak s_{-1}}{\mathfrak s_{-2}}=0,  
\end{aligned}$$
The negative part of $\gla s$ equipped with this scalar product becomes 
a pseudo H-type Lie algebra which is isomorphic to 
$\mathfrak H^{(2)}(\mathbb F,A,\gamma)$ as a pseudo H-type Lie algebra. 

\begin{enumerate}
\renewcommand{\labelenumi}{\textbf{Case \arabic{enumi}:} }
\item  
$\mathbb F=\mathbb C$. $\mathfrak s$ is equal to $\mathfrak{sl}(n+2,\mathbb C)_\mathbb R$. Hence the GLA $\gla s$ is a finite dimensional SGLA of type $(A_{\ell},\{\alpha_1,\alpha_\ell\})$ $(\ell=2n+1)$. 
If $\gamma=-1$ (resp. $\gamma=1$), then the signature of $\inph{\cdot}{\cdot}_{-2}$ is $(2,0)$ (resp. $(0,2)$). 

\item 
$\mathbb F=\mathbb C'$. Since $\mathbb C'$ is isomorphic to $\mathbb R\oplus 
\mathbb R$ as a $\mathbb R$-algebra, 
$\mathfrak s$ is isomorphic to 
$\mathfrak{sl}(n+2,\mathbb R)\times \mathfrak{sl}(n+2,\mathbb R)$. 
Hence the GLA $\gla s$ is a semisimple GLA of type 
$((\mathrm{AI})_{\ell},\{\alpha_1,\alpha_{\ell}\})\times( (\mathrm{AI})_\ell,\{\alpha_1,\alpha_{\ell}\})$, where $\ell=n+1$. 
The signature of $\inph{\cdot}{\cdot}_{-2}$ is $(1,1)$. 
\item
$\mathbb F=\mathbb H$. 
The GLA $\gla s$ is a finite dimensional SGLA of type $((\mathrm{AII})_{\ell},\{\alpha_2,\alpha_{\ell-1}\})$, 
where $\ell=2n+3$. 
If $\gamma=-1$ (resp. $\gamma=1$), then the signature of $\inph{\cdot}{\cdot}_{-2}$ is $(4,0)$ (resp. $(0,4)$). 
\item 
$\mathbb F=\mathbb H'$. 
Since $\mathbb H'$ is isomorphic to $M_2(\mathbb R)$ as a $\mathbb R$-algebra, 
$\mathfrak g$ is isomorphic to $\mathfrak{sl}(2n+2,\mathbb R)$. 
Hence the GLA $\gla s$ is a finite dimensional SGLA of type $((\mathrm{AI})_{\ell},\{\alpha_2,\alpha_{\ell-1}\})$, where $\ell=2n+3$. 
The signature of $\inph{\cdot}{\cdot}_{-2}$ is $(2,2)$. 
\end{enumerate}
Let $\mathcal H^{(2)}$ be the set of all equivalence classes of comH-type algebras 
$(\mathfrak n,\innerproduct)$ of 
the second class with $\dim\mathfrak n_{-2}\geqq3$. 
Let $\mathcal P^{(2)}$ be the set of all isomorphism classes (as a GLA) 
of real SGLAs given in Table \ref{table2}. 
We define a mapping $\varPhi_2$ of $\mathcal H^{(2)}$ into $\mathcal P^{(2)}$ 
as follows: 
for $[(\mathfrak n,\innerproduct)]\in \mathcal H^{(2)}$ we define 
$\varPhi_2([(\mathfrak n,\innerproduct)])$ by the equivalence class 
of the prolongation of $\mathfrak n$. 
From the above results, \cite{AS14:1}*{Theorem 3.6}, \cite{Gom96:1}*{\S3}, 
$\varPhi_2$ is well-defined and surjective. By \cite{AS14:1}*{Theorem 3.6}, 
$\varPhi_2$ is also injective. 
\begin{proposition} \label{prop43}
$\varPhi_2$ is bijective and the correspondence follows from Table \ref{table2}. 
\end{proposition} 
\begin{table}[hbtp]
\caption{Second class}
\label{table2}
\begin{tabular}{|c|c|c|c|c|c|}\hline
$\mathbb F$ & $\gamma$ & $\sgn(\innerproduct_{-2})$ & $A$ & $\mathfrak s$ & the gradation \\ \hline\hline
$\mathbb H$ &  $-1$ & $(4,0)$ & $S_{n,0}$ & $\mathfrak{sl}(n+2,\mathbb H)$ & 
$(\mathrm{(AII)}_\ell,\{\alpha_2,\alpha_{\ell-1}\})$ 
{\scriptsize $(\ell=2n+3)$} \\ \hline
$\mathbb H$ &  $1$ & $(0,4)$ & $S_{n,0}$ & $\mathfrak{sl}(n+2,\mathbb H)$ & 
$(\mathrm{(AII)}_\ell,\{\alpha_2,\alpha_{\ell-1}\})$ 
{\scriptsize $(\ell=2n+3)$} 
\\ \hline
$\mathbb H'$ & $-1$  & $(2,2)$ & $S_{n,0}$ & $\mathfrak{sl}(2n+4,\mathbb R)$ & 
$(\mathrm{(AI)}_\ell,\{\alpha_2,\alpha_{\ell-1}\})$ 
{\scriptsize ($\ell=2n+3)$} 
\\ \hline
$\mathbb O$  & $-1$ & $(8,0)$ & $1$ & $\mathrm{EIV}$ & 
$(\mathrm{EIV},\{\alpha_1,\alpha_6\})$ \\ \hline
$\mathbb O$  & $1$ & $(0,8)$ & $1$ & $\mathrm{EIV}$ & 
$(\mathrm{EIV},\{\alpha_1,\alpha_6\})$
\\ \hline
$\mathbb O'$ &  $-1$ & $(4,4)$ & $1$ & $\mathrm{EI}$& 
$(\mathrm{EI},\{\alpha_1,\alpha_6\})$ \\ \hline
\end{tabular}
\end{table}
\subsection{Matricial models of comH-type Lie algebras of the third class}\label{sec44}
Let $\mathfrak s$ be the simple Lie algebra $\mathfrak{su}(p+q,p)$ of rank $\ell=2p+q-1$. 
We define subspaces $\mathfrak s_p$ of $\mathfrak s$ as follows: 
$$\begin{aligned}
&\mathfrak s_{-2}
 =\{~
0_2 \ominus  0_{\ell-3} \ominus X_{31} 
\in\mathfrak s:X_{31}\in\mathfrak{gl}(2,\mathbb C),
X_{31}^*=-K_2X_{31}K_2
~\}, \\
&\mathfrak s_{-1}
=\left\{\begin{bmatrix}
0  & 0 &0 \\
X_{31}  & 0 &0 \\
0 & -X_{31}^*S_{p-2,q}  & 0  \\
\end{bmatrix}\in\mathfrak s:
X_{31}\in M(\ell-3,1,\mathbb C)\right\}, \\
& 
\mathfrak s_{0}=\left\{
X_{11} \oplus X_{22} \oplus(-K_2X_{11}^*K_2) \in\mathfrak s: 
\begin{aligned} 
& X_{11}\in\mathfrak{gl}(2,\mathbb C), 
X_{22}\in \mathfrak{gl}(\ell-3,\mathbb C), \\
& X_{22}+S_{p-2,q}X^*_{22}S_{p-2,q}=O
\end{aligned}
\right\}, \\
& 
\mathfrak s_p=\{~X\in\mathfrak s:
{}^tX\in\mathfrak s_{-p}~\}\quad (p=1,2), 
\quad \mathfrak s_p=\{0\} \quad (|p|>2). 
\\
\end{aligned}$$
\begin{lemma}\label{lem43}
Let $B$ be the Killing form on a finite dimensional SGLA $\gla g$ and $\sigma$ an involutive automorphism 
of $\mathfrak g$ such that $\sigma(\mathfrak g_p)=\mathfrak g_{-p}$. 
For a nonzero real number $\delta$ we define a nondegenerate symmetric 
bilinear form $(\cdot\mid \cdot)_\delta$ 
on $\mathfrak g$ as follows: 
$$(X\mid Y)_\delta=\delta B(\sigma(X),Y)$$
For $Z\in \mathfrak g_{-2}$ we define an endomorphism $J_Z$ of $\mathfrak g_{-1}$
by 
$$(J_Z(X)\mid Y)_\delta=(Z\mid [X,Y])_\delta\quad (X,Y\in\mathfrak g_{-1}).$$
Then 
$$J_Z=\ad(Z)\comp \sigma,\quad J_Z^2=\ad([Z,\sigma(Z)]).$$
\end{lemma}
\begin{proof} For $X,Y\in\mathfrak g_{-1}$
$$\begin{aligned}
(J_Z(X)\mid Y)_\delta&=(Z\mid [X,Y])_\delta=\delta B(\sigma(Z), [X,Y])
=\delta B(Z,\sigma([X,Y]) 
=\delta B(Z,[\sigma(X),\sigma(Y)]) \\
&=\delta B([Z,\sigma(X)],\sigma(Y))
=\delta B(\sigma[Z,\sigma(X)],Y)=([Z,\sigma(X)]\mid Y)_\delta
\end{aligned}$$
Hence $J_Z=\ad(Z)\comp \sigma$. Moreover 
$$J_Z^2=\ad(Z)\comp\sigma\comp\ad(Z)\comp\sigma=
\ad(Z)\comp\sigma^2\comp\ad(\sigma (Z))=\ad([Z,\sigma(Z)]). \qedhere$$
\end{proof}
\subsubsection{The case of signature $(1,3)$}
Let $\gla s$ be as in \S\ref{sec44}. 
If the negative part of $\gla s$ has the structure of a pseudo H-type Lie algebra, 
then a minimal admissible $\Cl(1,3)$ module is of dimensional 8(\cite{Cia00:1}*
{Theorem 3.2}). 
Since $\mathfrak s_{-1}$ is a 
$\Cl(\mathfrak s_{-2},\innerproduct_{-2})$-module, 
we get $4(2p+q-4)\equiv 0$ $\pmod{8}$, 
so $q$ is a even number.  

Now we assume that $p\geqq3$, $q=0$. 
We set 
$Q=J_2 \oplus  J_{p-2} \oplus J_2$ 
and 
define an involutive automorphism $\sigma$ of $\mathfrak s$ as follows:
$$\sigma(X)=Q\transpose{X}Q\qquad (X\in\mathfrak s).$$
Moreover we define a nondegenerate symmetric bilinear form on $\mathfrak s$ by 
$$\inph{X}{Y}=-\frac1{4p}B(\sigma(X),Y)\quad (X,Y\in\mathfrak s)$$
For $Z\in\mathfrak g_{-2}$ let $J_Z$ be the mapping 
of $\mathfrak g_{-1}$ to itself defined by 
$$\inph{J_Z(X)}{Y}=\inph{Z}{[X,Y]}\qquad (X,Y\in\mathfrak s_{-1}).$$ 
By Lemma \ref{lem43}, we obtain that 
$$J_Z^2(X)=-\inph{Z}{Z}X.$$ 
Then $(\mathfrak s_-,\inph{\cdot}{\cdot})$ becomes a pseudo H-type 
Lie algebra. 
Note that the restriction of $\inph{\cdot}{\cdot}$ on 
$\mathfrak s_{-2}$ has the signature $(1,3)$ and $\gla g$ 
is a finite dimensional SGLA of type 
$(\mathrm{(AIIIb)}_{\ell},\{\alpha_2,\alpha_{\ell-1}\})$, where $\ell=2p-1$. Moreover $(\mathfrak s_-,\inph{\cdot}{\cdot})$ is isomorphic to 
$\mathfrak H^{(3)}(\mathbb H',K_{p-2})$ as a pseudo H-type Lie algebra. 

\subsubsection{The case of signature $(3,1)$}
Let $\gla s$ be as in \S\ref{sec44}. 
If the negative part of $\gla s$ has the structure of a pseudo H-type Lie algebra, 
then $q$ is an even number. Indeed, since a minimal admissible $\Cl(3,1)$-module is of dimension 8 (\cite{Cia00:1}*{Theorem 3.2}),  
we get $4(2p+q-4)\equiv 0$ $\pmod{8}$, 
so $q$ is a even number.  

Now we assume that $p-2=2n$, $q=2m$, $m\geqq1$. 
We set 
$T=(-J_2) \oplus J_{p-2} \oplus I_{q} \oplus(-J_{p-2}) \oplus J_2$ and 
define an involutive automorphism $\sigma$ of $\mathfrak s$ as follows:
$$\sigma(X)=T\transpose{X}T\qquad (X\in\mathfrak g).$$
Moreover we define a nondegenerate symmetric bilinear form on $\mathfrak s$ by 
$$\inph{X}{Y}=-\frac1{4(2p+q)}B(\sigma(X),Y)\quad (X,Y\in\mathfrak s)$$
For $Z\in\mathfrak s_{-2}$ let $J_Z$ be the mapping 
of $\mathfrak s_{-1}$ to itself defined by 
$$\inph{J_Z(X)}{Y}=\inph{Z}{[X,Y]}\qquad (X,Y\in\mathfrak s_{-1}).$$ 
By Lemma  \ref{lem43}, we obtain that 
$$J_Z^2=-\inph{Z}{Z}1_{\mathfrak s_{-1}}.$$ 
Note that the signature of the restriction of $\innerproduct$ to 
$\mathfrak s_{-2}$ is $(3,1)$ and 
$\gla s$ is a finite dimensional SGLA of type $(\mathrm{(AIIIa)}_{\ell,p},\{\alpha_2,\alpha_{\ell-1}\})$, where $\ell=2p+q-1$. 
Moreover $(\mathfrak s_-,\inph{\cdot}{\cdot})$ is isomorphic to $\mathfrak H^{(3)}(\mathbb H,S_{n,m})$ as 
a pseudo H-type algebra. 

Let $\mathcal H^{(3)}$ be the set of all equivalence classes of comH-type algebras 
$(\mathfrak n,\innerproduct)$ of the third class. 
Let $\mathcal P^{(3)}$ be the set of all isomorphism classes (as a GLA) 
of real SGLAs given in Table \ref{table3}. 
We define a mapping $\varPhi_3$ of $\mathcal H^{(3)}$ into $\mathcal P^{(3)}$ 
as follows: 
for $[(\mathfrak n,\innerproduct)]\in \mathcal H^{(3)}$ we define 
$\varPhi_3([(\mathfrak n,\innerproduct)])$ by the equivalence class 
of the prolongation of $\mathfrak n$. 
From the above results, \cite{AS14:1}*{Theorem 3.6}, \cite{Gom96:1}*{\S3}, 
$\varPhi_2$ is well-defined and surjective. By \cite{AS14:1}*{Theorem 3.6}, 
$\varPhi_2$ is also injective. 
\begin{proposition} \label{prop44}
$\varPhi_3$ is bijective and the correspondence follows from Table \ref{table3}. 
\end{proposition} 
\begin{table}[hbtp]
\caption{Third class}
\label{table3}
\begin{tabular}{|c|c|c|c|c|}\hline
$\mathbb F$ & $\sgn(\innerproduct_{-2})$ & $A$ & $\mathfrak s$ & the gradation \\ \hline\hline
$\mathbb H$ & $(3,1)$ & $S_{n,m}$ 
& $\mathfrak{su}(q+p,p)$  & 
$(\mathrm{(AIIIa)}_{\ell,p},\{\alpha_2,\alpha_{\ell-1}\})$  
\\ 
 &   & {\footnotesize $(m\geqq1)$} 
& {\footnotesize $(p=2n-2,q=2m)$} & 
{\footnotesize $(\ell=2p+q-1)$} \\ \hline
$\mathbb H'$ & $(1,3)$ & $S_{n,0}$ 
& $\mathfrak{su}(p,p)$ & $(\mathrm{(AIIIb)}_{\ell},\{\alpha_2,\alpha_{\ell-1}\})$ \\ 
 & & {\footnotesize $(n\geqq1)$}& {\footnotesize $(p=2n-2)$} &  
{\footnotesize $(\ell=p+1)$}\\ \hline
$\mathbb O$ & $(7,1)$ & $1$ & $\mathrm{EIII}$ & $(\mathrm{EIII},\{\alpha_1,\alpha_6\})$ \\ \hline
$\mathbb O'$ & $(3,5)$ & $1$ & $\mathrm{EII}$ & $(\mathrm{EII},\{\alpha_1,\alpha_6\})$ \\ \hline
\end{tabular}
\end{table}
\subsection{Summary} 
Since a pseudo H-type algebra is an extended translation algebra, 
the prolongation of a pseudo H-type algebra is an SGLA appeared in the table in 
\cite{AS14:1}*{Theorem 3.6}. Comparing Tables 1--3 and the table in 
\cite{AS14:1}*{Theorem 3.6}, we obtain the following theorem.  
\begin{theorem} 
\begin{enumerate}
\fixitem Let $(\mathfrak n,\innerproduct)$ be a comH-type Lie algebra. Then 
$\mathfrak n$ is isomorphic to the negative part of some finite dimensional real semisimple GLA $\gla s$. In particular, if $\dim \mathfrak n_{-2}\ne2$, then $\mathfrak s$ is simple.  
\item 
Let $(\mathfrak n,\innerproduct)$ be a pseudo H-type Lie algebra and 
$\pcgla[n]g$ be the prolongation of $\mathfrak n$. 
Assume that: 
\textup{(i)} $\dim \mathfrak n_{-2}\geqq3$ and $\mathfrak g(\mathfrak n)_1\ne0$; 
\textup{(ii)} $\sgn(\innerproduct_{-2})\ne (1,3),(3,1)$. 
Then $\mathfrak n$ is isomorphic to some comH-type Lie algebra 
as a GLA. 
\end{enumerate}
\label{th42}
\end{theorem}
\section{Pseudo H-type Lie algebras satisfying the  $J^2$-condition}
Let $(\mathfrak n,\innerproduct)$ be a pseudo H-type Lie algebra. 
For any $x\in \mathfrak n_{-1}$ with $\inph{x}{x}\ne0$ we set 
$$J_{\mathfrak n_{-2}}(x)=
\{\,J_{z}(x):z\in\mathfrak n_{-2}\,\},\qquad 
\mathfrak n_{-1}(x)=\mathbb Rx+J_{\mathfrak n_{-2}}(x);
$$
then $\mathfrak n_{-1}(x)$ is a nondegenerate subspace of $\mathfrak n_{-1}$ 
with respect to $\innerproduct$. 
We say that  $(\mathfrak n,\innerproduct)$ satisfies the $J^2$ condition if 
for any $z\in \mathfrak n_{-2}$ and any $x\in \mathfrak n_{-1}$ with $\inph{x}{x}\ne0$, $\mathfrak n_{-1}(x)$ is $J_z$-stable. 
Clearly if $\dim \mathfrak n_{-2}=1$, then 
$(\mathfrak n,\innerproduct)$ satisfies the $J^2$-condition. 
If a pseudo H-type Lie algebra $(\mathfrak n,\innerproduct)$ 
is equivalent to a pseudo H-type Lie algebra $(\mathfrak n',\innerproduct')$  satisfying the $J^2$ condition, 
then $(\mathfrak n,\innerproduct)$ also satisfies one. 
For the cases $\mathbb F=\mathbb C,\mathbb C',\mathbb H,\mathbb H'$ it is clear that 
a comH-type Lie algebra $\mathfrak H^{(1)}(\mathbb F,S)$ of the first class satisfies $J^2$-condition. 
For the cases $\mathbb F=\mathbb O,\mathbb O'$ it follows from the following lemma. 
\begin{lemma} Let $(\mathfrak n,\innerproduct)$ be a pseudo H-type algebra. 
If $\dim\mathfrak n_{-2}+1=\dim\mathfrak n_{-1}$, then 
$(\mathfrak n,\innerproduct)$ satisfies the $J^2$ condition. 
\end{lemma} 
\begin{proof} 
Let $z_1,z_2$ be elements of $\mathfrak n_{-2}$ such that $\inph{z_1}{z_2}=0$ 
and let $x$ be an element of $\mathfrak n_{-1}$ such that $\inph{x}{x}\ne0$. 
Then $(\Ker\ad x|\mathfrak n_{-1})$ is a nondegenerate subspace of 
$(\mathfrak n_{-1},\innerproduct_{-1})$ and 
$$\mathfrak n_{-1}=(\Ker\ad x|\mathfrak n_{-1})\oplus (\Ker\ad x|\mathfrak n_{-1})^\perp.$$
We define a linear mapping $\varphi$ of $\mathfrak n_{-2}$ into 
$J_{\mathfrak n_{-2}}(x)$ 
as follows: 
$$\varphi(z)=J_z(x)\quad (z\in\mathfrak n_{-2}).$$
Then $\varphi$ is a linear isomorphism. 
Since $J_{\mathfrak n_{-2}}(x)=(\Ker\ad x|\mathfrak n_{-1})^\perp$ and since 
$\dim\mathfrak n_{-2}+1=\dim\mathfrak n_{-1}$, we get 
$\Ker\ad(x)|\mathfrak n_{-1}=\mathbb Rx$. 
Since $J_{z_1}J_{z_2}=-J_{z_2}J_{z_1}$, 
$$\inph{J_{z_1}J_{z_2}x}{x}= 
-\inph{J_{z_2}x}{J_{z_1}x}=\inph{x}{J_{z_2}J_{z_1}x}=-\inph{x}{J_{z_1}J_{z_2}x},$$
so $\inph{J_{z_1}J_{z_2}x}{x}=0$. This means that $J_{z_1}J_{z_2}x\in J_{\mathfrak n_{-2}}(x)$. Hence there exists an element $z_3$ of $\mathfrak n_{-2}$ such that 
$J_{z_1}J_{z_2}x=J_{z_3}x$. 
\end{proof}
Conversely we prove that a pseudo H-type algebra satisfying $J^2$ condition is isomorphic to a comH-type algebra of the first class. 
Let $(\mathfrak n,\innerproduct)$  be a pseudo H-type Lie algebra satisfying the 
$J^2$ condition. For $x\in \mathfrak n_{-1}$ with $\inph{x}{x}\ne0$ 
we set $\mathcal A_x=\mathbb R\times \mathfrak n_{-2}$; then $\mathcal A_x$ 
is a real vector space. We define a multiplicative operation $\underset{x}{*}$ on $\mathcal A_x$ as follows:  for $(\lambda_1,z_1),(\lambda_2,z_2)\in \mathcal A_x$, we put 
$$(\lambda_1,z_1)\underset{x}{*}(\lambda_2,z_2)=(\lambda_3,z_3),$$
where $(\lambda_3,z_3)$ is defined by 
$$(\lambda_11_{\mathfrak n_{-1}}+J_{z_1})(\lambda_2 1_{\mathfrak n_{-1}}+J_{z_2})x=(\lambda_3 1_{\mathfrak n_{-1}}+J_{z_3})x.$$ 
Then $(\mathcal A_x,+,\underset{x}{*})$ is an algebra over $\mathbb R$. 
We define an endomorphism $s$ of $\mathcal A_x$ as follows:  
$$s(\lambda,z)=(\lambda,-z);$$
then $s$ is an anti-involution of $\mathcal A_x$ and satisfies 
$$(\lambda,z)+s(\lambda,z)=(2\lambda,0)\in\mathbb R, 
\quad (\lambda,z)\underset{x}{*}s(\lambda,z)=(\lambda^2+\inph{z}{z},0)\in\mathbb R.$$ 
We define $N:\mathcal A_x\to\mathbb R$ as follows: 
$$N(\lambda,z)= (\lambda,z)\underset{x}{*}s(\lambda,z);$$ 
then $N$ is a non-degenerate quadratic form on $\mathcal A_x$ and hence 
$(\mathcal A_x,s)$ becomes a Cayley algebra. 

Furthermore we can prove that $\mathcal A_x$ becomes an alternative algebra and hence a normed algebra. By Hurwitz theorem (\cite{Har90:1}*{Theorem 6.37}), $\mathcal A_x$ is isomorphic to one of 
$\mathbb R,\mathbb C,\mathbb C',\mathbb H,\mathbb H',\mathbb O,\mathbb O'$ 
as a Cayley algebra. 
However since $\mathfrak n_{-2}\ne0$, $\mathcal A_x$ is not isomorphic to $\mathbb R$. 
Also the Cayley algebra $\mathcal A_x$ does not depend on the choice 
of the element $x$. 

We choose elements $x_1,\dots,x_{r+s}$ of $\mathfrak n_{-1}$ satisfying the following conditions: 
$$\begin{aligned}
&\inph{x_i}{x_i}=1\quad  (i=1,\dots,r),\qquad \inph{x_j}{x_j}=-1\quad (j=r+1,\dots,r+s), \\
&
\inph{\mathfrak n_{-1}(x_i)}{\mathfrak n_{-1}(x_j)}=0\quad  (i\ne j), \qquad 
\mathfrak n_{-1}=\mathfrak n_{-1}(x_1)\oplus\cdots\oplus\mathfrak n_{-1}(x_{r+s}).
\end{aligned}$$ 
In particular, if $\mathcal A_{x_i}$ is isomorphic to $\mathbb O$ or $\mathbb O'$ for some $i$, then $r+s=1$. 
We denote by $\mathbb F$ the Cayley algebra $\mathcal A_{x_1}$. 
We define a linear mapping $\varphi$ of $\mathfrak n$ to 
$\mathfrak h^{(1)}(\mathbb F,1_{r,s})=\mathbb F^{r+s}\oplus \ima \mathbb F$ as follows: 
$$\varphi\left(\sum_{i=1}^{r+s}(\lambda_ix_i+J_{z_i}(x_i))\right)
=((\lambda_1,z_1),\dots,(\lambda_{r+s},z_{r+s})) 
\ (\lambda_i\in\mathbb R,z_i\in\mathfrak n_{-2}),\quad 
\varphi(z)=-z \ (z\in\mathfrak n_{-2}). 
$$
Then $\varphi$ is an isomorphism as a pseudo H-type Lie algebra. 

\begin{theorem}
Let $(\mathfrak n,\innerproduct)$ be a pseudo H-type Lie algebra. 
The following three conditions are mutually equivalent: 
\begin{enumerate}
\renewcommand{\labelenumi}{\textup{(\roman{enumi})} }
\item $(\mathfrak n,\innerproduct)$ satisfies the $J^2$-condition; 
\item $(\mathfrak n,\innerproduct)$ is equivalent to a comH-type Lie algebra of the first class;
\item The CPSF associated with  $(\mathfrak n,\innerproduct)$ is of semisimple type. 
\end{enumerate}
In this case, the prolongation of $(\mathfrak n,[\innerproduct_{-1}])$ is an SGLA  
whose complexification is simple. 
\label{th51}
\end{theorem}
\begin{proof}
The equivalence (i) $\Leftrightarrow$ (ii) is obtained from the above results. 
The implication (ii) $\Rightarrow$ (iii) follows from Proposition \ref{prop42}. 
Finally we prove the implication (iii) $\Rightarrow$ (ii).  
We assume $(\mathfrak n,\innerproduct)$ satisfies the condition (iii). 
From the classification of the prolongations of 
CPSFs of semisimple type, the prolongation of 
$(\mathfrak n,[\innerproduct_{-1}])$ is isomorphic to 
the prolongation of the CPSF associated 
with some comH-type Lie algebra of the first class. 
Thus (iii) $\Rightarrow$ (ii) follows from Proposition \ref{prop22}.  
\end{proof}

\section{Prolongations of CPSFs associated with pseudo H-type Lie algebras}
Let $(\mathfrak n,\innerproduct)$ and $\pcgla[n]g$ be as in Proposition \ref{prop41}.  Moreover let $\gla g$ be the prolongation of $(\mathfrak n,[\innerproduct_{-1}])$. We define subspaces $\mathfrak h_0$, $\check{\mathfrak h}_0^a$ and 
$\check{\mathfrak h}_0^s$ of $\mathfrak g(\mathfrak n)_0$ as follows:  
$$\begin{aligned}
& \mathfrak h_0=\check{\mathfrak h}_0\cap\mathfrak g_0, \\
& \check{\mathfrak h}_0^a=
\{~D\in\check{\mathfrak h}_0:\inph{[D,x]}{y}+\inph{x}{[D,y]}=0
\quad \text{for all}\ x,y\in \mathfrak n_{-1}~\}, \\
& \check{\mathfrak h}_0^s=
\{~D\in\check{\mathfrak h}_0:\inph{[D,x]}{y}-\inph{x}{[D,y]}=0
\quad \text{for all}\ x,y\in \mathfrak n_{-1}~\}, \\
\end{aligned}$$
\begin{proposition}  
Under the above assumptions, 
$$\check{\mathfrak h}_0=\check{\mathfrak h}_0^a\oplus \check{\mathfrak h}_0^s,\quad 
\mathfrak h_0=\check{\mathfrak h}_0^a, \quad 
\mathfrak g_0=\mathfrak{so}(\mathfrak n_{-2},\innerproduct_{-2})\oplus\mathbb RE\oplus \check{\mathfrak h}_0^a$$ 
\label{prop61}
\end{proposition}
\begin{proof}
Since $D^\top\in\check{\mathfrak h}_0$ for 
$D\in\check{\mathfrak h}_0$, we get 
$\check{\mathfrak h}_0=\check{\mathfrak h}_0^a\oplus\check{\mathfrak h}_0^s$, so $\mathfrak h_0=\check{\mathfrak h}_0^a$. From Proposition \ref{prop41} 
the last assertion is obvious. 
\end{proof}
\begin{lemma}
Let $(\mathfrak n,\innerproduct)$ be a pseudo H-type Lie algebra, 
and let $\gla{g}$ be the prolongation of $(\mathfrak n,[\innerproduct_{-1}])$. 
For $p\geqq1$, the condition ``$x\in\mathfrak g_p$ and 
$[x,\mathfrak g_{-2}]=0$'' implies $x=0$. 
\label{lem61}
\end{lemma}
\begin{proof}
We identify $\mathfrak h_0$ with a subspace of $\mathfrak{gl}(\mathfrak n_{-1})$.  
For a subspace $\mathfrak a$ of $\mathfrak{gl}(\mathfrak n_{-1})$ 
we denote by $\rho^{(k)}(\mathfrak a)$ the $k$-th (algebraic) prolongation of 
$\mathfrak a$. 
By Proposition \ref{prop61}, 
$\mathfrak h_0\subset \mathfrak{so}(\mathfrak n_{-1},\innerproduct_{-1})$; 
hence 
$\rho^{(1)}(\mathfrak h_0)\subset\rho^{(1)}(\mathfrak{so}(\mathfrak n_{-1},\innerproduct_{-1}))=0$. 
The lemma is proved. 
\end{proof}
\begin{theorem} \label{th61}
Let $(\mathfrak n,\innerproduct)$ be a pseudo H-type algebra with $\dim\mathfrak n_{-2}\geqq3$. Let $\gla g$ be the prolongation of $(\mathfrak n,[\innerproduct_{-1}])$. 
If $\mathfrak g_1\ne0$, then $\gla g$ is simple. 
\end{theorem} 
\begin{proof} The proof is a variant on the proof of Proposition 3.3 in \cite{AS14:2}. 
Let $\mathfrak r$ be the radical of $\mathfrak g$ and $\pcgla[n]g$ be the prolongation of $\mathfrak n$. 
We assume that there exists a positive integer $k$ such that 
$\mathcal D^k\mathfrak r\ne0$ and $\mathcal D^{k+1}\mathfrak r=0$, 
where $(\mathcal D^r\mathfrak r)_{r\in\mathbb Z_{\geqq1}}$ is the derived series of $\mathfrak r$. 
Then $\mathcal D^k\mathfrak r$ is a commutative graded ideal of $\mathfrak g$; 
$\mathcal D^k\mathfrak r=\bigoplus\limits_{p\in\mathbb Z}\mathfrak q_p$, 
$\mathfrak q_p=\mathcal D^k\mathfrak r\capprod \mathfrak g_p$. 
Since the $\mathfrak g_0$-module $\mathfrak g_{-2}$ is irreducible (Proposition \ref{prop61}), 
we get $\mathfrak q_{-2}=0$ or $\mathfrak q_{-2}=\mathfrak n_{-2}$. 
Since $\mathfrak n$ is nondegenerate, if $\mathfrak q_{-2}=0$, then 
$\mathfrak q_{-1}=0$. By transitivity, we see that $\mathcal D^k\mathfrak r=0$,  
which is a contradiction. 
Hence $\mathfrak q_{-2}=\mathfrak n_{-2}$. 
If $\mathfrak q_{-1}=0$, then 
$[\mathfrak g_1,\mathfrak g_{-2}]\subset \mathfrak q_{-1}=0$. 
By Lemma \ref{lem61}, 
we get $\mathfrak g_1=0$, which is a contradiction. 
If $\mathfrak q_{-1}=\mathfrak g_{-1}$, then 
$\mathfrak g_{-2}=[\mathfrak q_{-1},\mathfrak q_{-1}]\subset \mathcal D^{k+1}\mathfrak r=0$, 
which is a contradiction. 
Hence $\mathfrak q_{-1}$ is a proper subspace of $\mathfrak n_{-1}$. 
For every non-isotropic $z\in\mathfrak n_{-2}$ 
there exists a grade-preserving automorphism $\psi_z$ of $\mathfrak g(\mathfrak n)$ such that 
$\psi_z(s)=J_z(s)$ for all $s\in \mathfrak n_{-1}$ (\cite{AS14:1}*{Proposition 2.6}). 
Since $\psi_z(D)\in \mathfrak g_0$ for any $D\in\mathfrak g_0$,  
$\psi_z$ induces a grade-preserving automorphism of $\mathfrak g$, which is denoted by the same letter. 
Since $\mathfrak r$ is a characteristic ideal of $\mathfrak g$, 
$\psi_z(\mathfrak q_{-1})=\mathfrak q_{-1}$ for all non-isotropic $z\in\mathfrak n_{-2}$, so 
$\mathfrak q_{-1}$ is a $\Cl(\mathfrak n_{-2},\innerproduct_{-2})$-submodule of 
$\mathfrak n_{-1}$. 
It follows that $\mathfrak q_{-1}$ is a $\innerproduct_{-1}$-isotropic 
$\Cl(\mathfrak n_{-2},\innerproduct_{-2})$-submodule of $\mathfrak n_{-1}$ and that 
$$\mathfrak q^\perp_{-1}=\{~s\in\mathfrak g_{-1}:
\inph{s}{\mathfrak q_{-1}}=0~\}
=\{~s\in\mathfrak g_{-1}:[s,\mathfrak q_{-1}]=0~\}$$
is a proper $\Cl(\mathfrak n_{-2},\innerproduct_{-2})$-submodule of $\mathfrak n_{-1}$ 
containing $\mathfrak q_{-1}$. 

Let $\mathfrak a$ be a $\Cl(\mathfrak n_{-2},\innerproduct_{-2})$-submodule of $\mathfrak n_{-1}$ which is the complementary to $\mathfrak q_{-1}^\perp$. 
Since the restriction $\eta$ of $\innerproduct_{-1}$ to $\mathfrak q_{-1}\times \mathfrak a$ is nondegenerate, 
we obtain the following decomposition of $\Cl(\mathfrak n_{-2},\innerproduct_{-2})$-submodules
$$\mathfrak n_{-1}=\mathfrak q_{-1}\oplus \mathfrak b\oplus \mathfrak a,$$
where 
$$\mathfrak b=\{~s\in\mathfrak g_{-1}:
\inph{s}{\mathfrak a+\mathfrak q_{-1}}=0~\}
=\{~s\in\mathfrak g_{-1}:[s,\mathfrak a+\mathfrak q_{-1}]=0~\}.$$
Let $\xi$ (resp. $\zeta$) be the restriction of $\innerproduct_{-1}$ to 
$\mathfrak a\times \mathfrak a$ (resp. $\mathfrak b\times \mathfrak b$). 
We denote by 
$$\eta^\flat:\mathfrak q_{-1}\to\mathfrak a^*,\quad 
\xi^\flat:\mathfrak a\to\mathfrak a^*,\quad 
\zeta^\flat:\mathfrak b\to\mathfrak b^*$$
the induced linear mappings defined by
$$\langle\eta^\flat(s),a\rangle=\inph{s}{a}\ (s\in\mathfrak q_{-1},a\in\mathfrak a), \ 
\langle\xi^\flat(s),a\rangle=\inph{s}{a}\ (s,a\in\mathfrak a), \ 
\langle\zeta^\flat(s),a\rangle=\inph{s}{a}\ (s,a\in\mathfrak b).$$ 
Then $\eta^\flat$ is a linear isomorphism and 
$$\eta^\flat\comp\psi_z=-\psi_z^*\comp\eta^\flat, \quad
\xi^\flat\comp\psi_z=-\psi_z^*\comp\xi^\flat, \quad
\zeta^\flat\comp\psi_z=-\psi_z^*\comp\zeta^\flat$$
for any non-isotropic $z\in\mathfrak n_{-2}$. 
This implies that the linear mapping 
$$\varphi=(\eta^\flat)^{-1}\comp \xi^\flat:\mathfrak a\to\mathfrak q_{-1}$$
is a $\Cl(\mathfrak n_{-2},\innerproduct_{-2})$-module homomorphism and 
that the $\Cl(\mathfrak n_{-2},\innerproduct_{-2})$-module 
$\mathfrak a'=\{~2a-\varphi(a):a\in\mathfrak a~\}$ is 
$\innerproduct_{-1}$-isotropic and the complementary to $\mathfrak q_{-1}^\perp$. 
Replacing $\mathfrak a$ with $\mathfrak a'$, if necessary, we may assume that 
$\mathfrak a$ is $\innerproduct_{-1}$-isotropic and commutative. 

Let $\varPhi$ be a nondegenerate bilinear form on $\mathfrak a$ 
such that 
$$\varPhi(J_zs,t)=\tau\varPhi(s,J_zt)=\sigma\varPhi(t,J_zs),$$
where $s,t\in\mathfrak a$, $z\in\mathfrak n_{-2}$, $\tau,\sigma\in\{\pm1\}$, 
$\tau\sigma=-1$. 
Such a $\varPhi$ does exist (see the proof of \cite{AS14:1}*{Theorem 3.6} and \cite{AC97:1}). 
We denote by 
$$\varPhi^\flat:\mathfrak a\to\mathfrak a^*$$
the induced linear mapping defined by 
$$\langle\varPhi^\flat(s),a\rangle=\varPhi(s,a)\quad (s,a\in\mathfrak a).$$ 
We define a linear mapping $\chi$ of $\mathfrak n$ into itself as follows: 
$$\chi|\mathfrak n_{-2}=1_{\mathfrak n_{-2}},\quad 
\chi|\mathfrak a=(\eta^\flat)^{-1}\comp \varPhi^\flat, \quad 
\chi|\mathfrak b=1_{\mathfrak b}, \quad 
\chi|\mathfrak q_{-1}=(\varPhi^\flat)^{-1}\comp \eta^\flat.$$ 
Then $\chi$ is a grade-preserving automorphism of $\mathfrak n$ and 
is isometry. Moreover 
$\chi$ is naturally extended to a grade-preserving automorphism of $\pcgla[n]g$, which is denoted by the same letter. Since $\chi$ is isometry, $\chi(\mathfrak g)=\mathfrak g$. 
Therefore $\chi$ induces a grade-preserving automorphism of $\gla g$, which is denoted by the same letter.  However since 
$\mathfrak a=\chi(\mathfrak q_{-1})\subset \mathcal D^k\mathfrak r$, we reach a contradiction. Thus we see that $\gla g$ is semisimple. 
\end{proof}

\begin{theorem} \label{th62}
Let $(\mathfrak n,\innerproduct)$ be a pseudo H-type Lie algebra, 
and let $\gla g$ be the prolongation of the associated CPSF 
$(\mathfrak n,[\innerproduct_{-1}])$. 
\begin{enumerate}
\item If $\dim \mathfrak n_{-2}=1$, then 
$\gla g$ is one of finite dimensional SGLAs of types 
$$(\mathrm{(AI)}_\ell,\{\alpha_1,\alpha_\ell\}), 
(\mathrm{(AIIIa)}_{\ell,p},\{\alpha_1,\alpha_\ell\}),  
(\mathrm{(AIIIb)}_\ell,\{\alpha_1,\alpha_\ell\}),  
(\mathrm{(AIV)}_\ell,\{\alpha_1,\alpha_\ell\}).$$ 
\item If $\dim \mathfrak n_{-2}=2$, then $\mathfrak g_1=0$. 

\item Assume that $\dim\mathfrak n_{-2}\geqq3$. 
If $\mathfrak g_1\ne0$, then 
$\gla g$ is a finite dimensional SGLA and coincides with the prolongation of $\mathfrak n$.  
Furthermore for 
$\mathfrak g_1$ to be nonzero, it is necessary and sufficient 
that $(\mathfrak n,\innerproduct)$ is a comH-type Lie algebra of the first class. Consequently, if $\mathfrak g_1$ is nonzero, then 
$\gla g$ is an SGLA of one of the following types: 
$$(\textup{(CIIa)}_{\ell,p},\{\alpha_2\}), 
(\textup{(CIIb)}_{\ell},\{\alpha_2\}),  
(\textup{(CI)}_\ell,\{\alpha_2\}), 
 (\mathrm{FII},\{\alpha_4\}), (\mathrm{FI},\{\alpha_4\})$$
\end{enumerate}
\end{theorem}
\begin{proof}
(1) Since $\dim \mathfrak n_{-2}=1$, the pseudo H-type Lie algebra $\mathfrak n$ 
satisfies the $J^2$-condition. Hence (1) follows from Theorem \ref{th51} and 
Proposition \ref{prop42}. 

(2) If $\mathfrak g$ is semisimple, then $\dim\mathfrak g_{-2}\ne2$ 
(Theorem \ref{th51}). Hence $\mathfrak g$ is not semisimple. 
Now we assume that $\sgn(\innerproduct_{-2})=(1,1)$. 
We use the notation in \ref{sec411} (2b). 
By Proposition \ref{prop61},  
$$\mathfrak g_0=\mathbb RE\oplus\mathbb RI\oplus\mathfrak h_0$$
and 
$$\mathfrak h_0=\{~D-D^\top:D\in\mathfrak g(\mathfrak n)^+_0\cap 
\check{\mathfrak h}_0~\},$$
where $D^\top$ is the adjoint of $D$ with respect to $\innerproduct$. 
Indeed, an element $D\in \mathfrak h_0$ is decomposed as follows: 
$D=D_1+D_2$, $D_1\in\mathfrak g(\mathfrak n)^+_0$, 
 $D_2\in\mathfrak g(\mathfrak n)^-_0$. Since 
$D=-D^\top$, we get $D_2=-D_1^\top$. 
 Since $\mathfrak g(\mathfrak n)^\pm$ are contact algebras, 
the correspondence $D\mapsto D_1$ induces an isomorphism of the ideal $\mathfrak h_0$ of $\mathfrak g_0$ onto $\mathfrak{sp}(\mathfrak n^+_{-1})$. 
Therefore the $\mathfrak g_0$-module $\mathfrak g_{-1}$ is completely reducible 
and the semisimple part of $\mathfrak g_0$ coincides with $\mathfrak h_0$. 
Let $\mathfrak r$ be the radical of $\mathfrak g$; 
then $\mathfrak r$ is a graded ideal of $\gla g$: 
$\gla r$, $\mathfrak r_p=\mathfrak g_p\capprod \mathfrak r$. 
Since the $\mathfrak g_0$-module $\mathfrak g_{-1}$ is completely reducible, 
there exists a graded Levi subalgebra $\gla s$ of $\mathfrak g$ 
such that $\mathfrak s_p=\mathfrak g_p$ for $p\geqq1$, 
$[\mathfrak r_0,\mathfrak s_{-1}]=0$ and $[\mathfrak r_{-1},\mathfrak s_1]=0$; 
then $\mathfrak h_0\subset \mathfrak s_0$.  
Now we assume that $\mathfrak g_1\ne0$ and 
$\{0\}\subsetneqq\mathfrak r_{-1}\subsetneqq\mathfrak g_{-1}$. 
Since $\mathfrak g_1\ne0$, the $\mathfrak g_0$-modules 
$\mathfrak g(\mathfrak n)^\pm_{-1}$ are not isomorphic. 
Therefore we may assume that $\mathfrak s_{-1}=\mathfrak g(\mathfrak n)^+_{-1}$ and 
$\mathfrak r_{-1}=\mathfrak g(\mathfrak n)^-_{-1}$. 
Let $\mathfrak a$ be a semisimple ideal of $\mathfrak g_0$ such that 
$[\mathfrak a,\mathfrak s_{-1}]=0$. 
Then $\mathfrak a\subset \mathfrak h_0$ and hence 
$\mathfrak a\subset \mathfrak g(\mathfrak n)^-_{-1}\capprod \mathfrak h_0=0$. 
Thus we get $\mathfrak s_0=[\mathfrak s_{-1},\mathfrak s_1]$. 
Since $[\mathfrak g(\mathfrak n)^-_{-1},\mathfrak s_1]=0$, we obtain 
$\mathfrak s_1\subset \mathfrak g(\mathfrak n)^+_{1}$. 
We see that 
$$\mathfrak h_0\subset\mathfrak s_0=[\mathfrak s_{-1},\mathfrak s_1]\subset 
\mathfrak g(\mathfrak n)^+_{0},$$
and hence 
$$[\mathfrak h_0,\mathfrak g(\mathfrak n)^-_{-1}]\subset 
[\mathfrak g(\mathfrak n)^+_{0},\mathfrak g(\mathfrak n)^-_{-1}]=0,$$
which is a contradiction. 
Hence $\mathfrak r_{-1}=0$ or $\mathfrak r_{-1}=\mathfrak g_{-1}$. 
If $\mathfrak r_{-1}=0$, then 
$$\mathfrak g_{-2}=[\mathfrak g_{-1},\mathfrak g_{-1}]
=[\mathfrak s_{-1},\mathfrak s_{-1}]=\mathfrak s_{-2},$$ 
which is a contradiction. 
Finally we assume $\mathfrak r_{-1}=\mathfrak g_{-1}$. 
Since $\dim\mathfrak s_{-1}=\dim\mathfrak s_{1}$, we obtain 
$\mathfrak g_1=\mathfrak s_1=0$. 

In case $\sgn(\innerproduct_{-2})=(2,0)$ or $(0,2)$ we can prove 
$\mathfrak g_1=0$ similarly by considering the complexification. 

(3) Assume that $\dim\mathfrak n_{-2}\geqq3$ and $\mathfrak g_1\ne0$. 
Then $\mathfrak g(\mathfrak n)_1\ne0$. 
By Theorem \ref{th61}, $\mathfrak g$ is simple. 
For $p>0$ we see 
$$\dim \mathfrak g_p=\dim \mathfrak g_{-p}=\dim \mathfrak g(\mathfrak n)_{-p}
=\dim \mathfrak g(\mathfrak n)_{p}$$
and hence 
$$\mathfrak g_p=\mathfrak g(\mathfrak n)_{p},\quad 
\mathfrak g_0=[\mathfrak g_1,\mathfrak g_{-1}]
=[\mathfrak g(\mathfrak n)_1,\mathfrak g(\mathfrak n)_{-1}]=\mathfrak g(\mathfrak n)_0.$$
The second and last assertions follow from Theorem \ref{th51} and Table 1. 
\end{proof}

\section{Einstein spaces associated with pseudo H-type Lie algebras}
Let $(\mathfrak g,\innerproduct)$ be a metric Lie algebra. 
Here a metric Lie algebra means a Lie algebra with a scalar product (a nondegenerate symmetric bilinear form). 
Let $G$ be a simply connected Lie group with Lie algebra $\mathfrak g$ and let $g$ be the left-invariant pseudo-riemannian metric induced by 
$\innerproduct$. 
Let $\nabla$ be the Levi-Civita connection on $(G,g)$.  
For $X,Y\in\mathfrak g$, the Koszul formula gives 
$$\nabla_XY=\frac12([X,Y]-(\ad Y)^*X-(\ad X)^*Y), $$
where $(\ad X)^*$ is the adjoint operator of $\ad X$ with respect to $\innerproduct$. 
The curvature tensor is defined by 
$$R(X,Y)Z=\nabla_X\nabla_YZ-\nabla_Y\nabla_XZ-\nabla_{[X,Y]}Z.$$
For $X,Y\in\mathfrak g$ we define the Ricci curvature $\ric^\mathfrak g$ 
by 
$$\ric^\mathfrak g(X,Y)=\tr(Z\mapsto R(Z,X)Y).$$
The Ricci operator $\Ric^\mathfrak g:\mathfrak g\to\mathfrak g$ defined by
$$\langle \Ric^\mathfrak g(X)\mid Y\rangle=\ric^\mathfrak g(X,Y).$$
A metric Lie algebra $(\mathfrak g,\innerproduct)$ is called an algebraic Ricci soliton if 
$$\Ric^\mathfrak g=\lambda 1_\mathfrak g+D,$$
where $\lambda\in\mathbb R$, $D\in\Der(\mathfrak g)$. 
In particular, an algebraic Ricci soliton on a nilpotent Lie algebra is called a nilsoliton. 
\subsection{Ricci tensor}
Let $(\mathfrak n,\innerproduct)$ be a pseudo H-type Lie algebra. 
By a pseudo $H$-type Lie group $(\mathfrak n,\innerproduct)$ we mean the corresponding simply connected Lie group $N$ equipped with the induced left-invariant pseudo-riemannian metric $g$. 
By \cite{Rya21:1}*{Theorems 3.5 and 3.6} and \cite{BO13:1}*{p.13}, we obtain 
\begin{theorem} \label{th71}
Let $(\mathfrak n,\innerproduct)$ be a pseudo $H$-type Lie algebra. 
The Ricci operator preserves the decomposition 
$\mathfrak n=\mathfrak n_{-1}\oplus\mathfrak n_{-2}$, 
and is given on each factor by 
$$\Ric^\mathfrak n\mid_{\mathfrak n_{-1}}=
-\frac12(\dim\mathfrak n_{-2})1_{\mathfrak n_{-1}},\qquad 
\Ric^\mathfrak n\mid_{\mathfrak n_{-2}}=\frac14(\dim\mathfrak n_{-1})1_{\mathfrak n_{-2}}$$
Furthermore 
the scalar curvature is constant on $N$, and is given by 
$$\scc^\mathfrak n=-\frac14(\dim \mathfrak n_{-1})(\dim \mathfrak n_{-2}).$$
 \end{theorem}
\subsection{Nilsolitons} 
Let $(\mathfrak n,\innerproduct)$ be a pseudo H-type Lie algebra. 
We put 
$$c=\frac14\dim\mathfrak n_{-1}+\dim\mathfrak n_{-2}=\frac14\tr((\ad E)^2),$$
where $E$ is the characteristic element of the prolongation of the FGLA 
$\mathfrak n$. 
By Theorem \ref{th71}, $D=\Ric^\mathfrak n+c1_\mathfrak n$ is a derivation. 
More precisely, 
$$D=\Ric^\mathfrak n+c1_\mathfrak n=\xi E,\quad 
\xi:=-\frac14\left(\dim\mathfrak n_{-1}+2\dim\mathfrak n_{-2}\right)=\frac14\tr(\ad E).$$
Hence we obtain the following proposition:
\begin{proposition}[\cite{Rya21:1}*{Corollary 4.12}]
Every pseudo H-type Lie algebra is a nilsoliton. 
\end{proposition}
\begin{lemma} If 
$\Ric^\mathfrak n=\lambda 1_{\mathfrak n}+D'$, $D'\in\Der(\mathfrak n)$, 
then $\lambda=-c$ and $D'=D=\ad(\xi E)$. 
\end{lemma}
\begin{proof} 
For $u,v\in\mathfrak n_{-1}$, 
$$D'([u,v])=\Ric^\mathfrak n([u,v])-\lambda[u,v]=
\left(\frac14\dim\mathfrak n_{-1}-\lambda\right)[u,v]$$
Since $D'$ is a derivation, 
$$D'[u,v]=[D'u,v]+[u,D'v]=
[\Ric^\mathfrak nu,v]+[u,\Ric^\mathfrak nv]-2\lambda[u,v]
=(-\dim\mathfrak n_{-2}-2\lambda)[u,v]$$
Hence 
$$\lambda=-\frac14\dim\mathfrak n_{-1}-\dim\mathfrak n_{-2}=-c$$
and $D'=D=\ad(\xi E)$. 
\end{proof}

\subsection{Extensions of nilsolitons}
Let $(\mathfrak n,\innerproduct)$ be a pseudo H-type Lie algebra and 
$\gla g$ be the prolongation of $\mathfrak n$. 

Moreover let $\mathfrak a$ be a maximal subalgebra of $\mathfrak g$ 
satisfying the following condition: 
(i) $\mathfrak a$ contains $E$ and is commutative; (ii) $\mathfrak a$ is $\mathbb R$-diagonalizable. 
Clearly $\mathfrak a\subset\mathfrak g_0$. 
Let  $\mathfrak a'$ be the centralizer of $\mathfrak g_0$ in $\mathfrak a$.  
We denote by $\mathfrak a''$ the set of all self-adjoint elements of 
$\mathfrak a'$ with respect to $\innerproduct$. 
Here $H\in\mathfrak a'$ is said to be self-adjoint with respect to 
$\innerproduct$ if  $\langle\ad H(X)\mid Y\rangle=
\langle X\mid\ad H(Y)\rangle$ for all $X,Y\in\mathfrak n$. 
By Propositions \ref{prop41} and \ref{prop61}, 
$\mathfrak a''=\mathbb RE\oplus \mathfrak a''_0$, 
$\mathfrak a''_0=\mathfrak a'\cap \check{\mathfrak h}_0^s$. 
Recall that $\End(\mathfrak n)$ has a natural symmetric bilinear form defined 
by 
$$(X,Y)_{\tr}=\tr(X\comp Y)\quad (X,Y\in\End(\mathfrak n)).$$ 
We denote by $\mathfrak a_\mathfrak n$ a maximal 
subspace of $\mathfrak a''$ satisfying the following two conditions: 
(i) $E\in\mathfrak a_\mathfrak n$; (ii) the restriction of $(\cdot\mid\cdot)_{\tr}$ to $\ad\mathfrak a_\mathfrak n$ is nondegenerate. 
\begin{remark} If $\gla g$ is the prolongation of the CPSF $(\mathfrak n,[\innerproduct_{-1}])$, 
then $\mathfrak a''=\mathbb RE$.  
\end{remark}
We now assume that 
the prolongation $\gla g$ of $\mathfrak n$ is a finite dimensional SGLA . 
Following Tamaru \cite{Tam11:1}*{\S3} we describe the gradation of $\mathfrak g$ by using the restricted root system. 
There exists a Cartan decomposition $\mathfrak g=\mathfrak k+\mathfrak p$ 
such that $\mathfrak a\subset \mathfrak p$. 
This Cartan decomposition defines an involutive automorphism $\sigma$ of 
$\mathfrak g$, which is called a Cartan involution, such that $\sigma(E)=-E$ and $\sigma(\mathfrak g_p)=\mathfrak g_{-p}$. 
The Cartan involution $\sigma$ defines a positive definite inner product $B_\sigma$ on $\mathfrak g$ by 
$$B_\sigma(X,Y)=-B(X,\sigma(Y)),\qquad X,Y\in\mathfrak g.$$ 
In the usual way, $\mathfrak a$ defines the restricted root system 
$\varDelta=\varDelta(\mathfrak g,\mathfrak a)$ with respect to 
$\mathfrak a$. 
Denoting by $\mathfrak g^\alpha$ the root space of a root $\alpha$, we obtain 
the root space decomposition 
$$\mathfrak g=\mathfrak g^0+\sum_{\alpha\in\varDelta}\mathfrak g^\alpha,
$$
where $\mathfrak g^0$ is the centralizer of $\mathfrak a$ in $\mathfrak g$. 
Let $(h_1,\dots,h_r)$ be a basis of $\mathfrak a$ such that $h_1=E$. 
We introduce the lexicographic ordering of $\varDelta\cup\{0\}$ with respect to the basis $(h_1,\dots,h_r)$. 
Let $\varDelta^+$ be the set of positive roots of $\varDelta$ with respect to this ordering. 
For  $H_1,H_2\in\mathfrak a''$  
$$(\ad H_1,\ad H_2)_{\tr}
=\sum_{\alpha\in\varDelta^+}(\dim\mathfrak g^\alpha)\alpha(H_1)\alpha(H_2)=\frac12 B_\sigma(H_1,H_2).$$
Thus we see that  $(\cdot\mid\cdot)_{\tr}$ is positive definite on 
$\ad \mathfrak a''$. In this case $\mathfrak a''=\mathfrak a_\mathfrak n$. 

Returning to the general case, we set $\mathfrak s=\mathfrak n\oplus \mathfrak a_\mathfrak n$ and define a nondegenerate symmetric bilinear form $\innerproduct_{\mathfrak s}$ on 
$\mathfrak s$ by 
$$\begin{aligned} 
\scalar{X}{Y}_\mathfrak s&=\scalar{X}{Y}\quad (X,Y\in\mathfrak n),\quad 
\scalar{X}{A}_\mathfrak s=0\quad (X\in\mathfrak n,A\in\mathfrak a_\mathfrak n),\\ 
\scalar{A}{B}_\mathfrak s&=\frac1c(\ad A,\ad B)_{\tr}\quad (A,B\in\mathfrak a_\mathfrak n), 
\end{aligned}$$
where for $A\in\mathfrak a_{\mathfrak n}$ 
we denote by $\ad A$ an endomorphism $\ad_\mathfrak s(A)|_\mathfrak n$ of $\mathfrak n$.  
Then $\mathfrak n=[\mathfrak s,\mathfrak s]$ and $\mathfrak a_\mathfrak n=\mathfrak n^\perp$. 
The solvable metric Lie algebra $(\mathfrak s,\innerproduct_\mathfrak s)$ 
constructed above is called a pseudo-Iwasawa extension of the pseudo H-type Lie algebra $(\mathfrak n,\innerproduct)$.  
We call $H_0\in\mathfrak s$ the mean curvature vector of 
$(\mathfrak s,\innerproduct_\mathfrak s)$ if 
$\scalar{H_0}{X}_\mathfrak s=\tr(\ad X)$ for all $X\in\mathfrak s$. 
By \cite{CR22:1}*{Lemma 1.15}, $H_0\in\mathfrak a_\mathfrak n$. 
By \cite{CR22:1}*{Lemma 3.8}, the Ricci tensor $\ric^\mathfrak s$ 
of $(\mathfrak s,\innerproduct_\mathfrak s)$ and its restriction to $\mathfrak n$ are related by: 
$$\begin{aligned} 
\ric^\mathfrak s(X,Y)&=\ric^\mathfrak n(X,Y)-\scalar{[H_0,X]}{Y} 
\quad (X,Y\in\mathfrak n), \\
\ric^\mathfrak s(X,A)&=0\quad 
(X\in\mathfrak n, A\in\mathfrak a_\mathfrak n),\\
\ric^\mathfrak s(A,B)&
=-(\ad A,\ad B)_{\tr}\quad 
(A,B\in\mathfrak a_\mathfrak n).
\end{aligned}$$
By using the first equation, for $X,Y\in\mathfrak n$, we get 
$$\begin{aligned}
-c\scalar{X}{Y}&=\scalar{\Ric^\mathfrak n(X)}{Y}
-\scalar{DX}{Y} 
=\scalar{\Ric^\mathfrak s(X)}{Y}+\scalar{\ad(H_0)X}{Y}-\scalar{DX}{Y} \\
&=\scalar{(\Ric^\mathfrak s\mid_\mathfrak n-D-\ad(H_0))X}{Y}.
\end{aligned}$$
By using the third  equation, for $A,B\in\mathfrak a_\mathfrak n$ we get 
$$-c\scalar{A}{B}=-(\ad A,\ad B)_{\tr}=\ric^\mathfrak s(A,B).$$
Hence 
$$\Ric^{\mathfrak s}\mid_\mathfrak n=D-\ad H_0-c1_\mathfrak n,\quad 
\Ric^{\mathfrak s}\mid_{\mathfrak a_\mathfrak n}
=-c1_{\mathfrak a_\mathfrak n}.$$
By \cite{CR22:1}*{the proof of Theorem 2.1, page 88}, 
$$(\Ric^\mathfrak n,\ad A)_{\tr}=0\quad \text{for all} \ A\in\mathfrak a_\mathfrak n. $$
Hence we see that 
$$\begin{aligned}
0&=(\Ric^\mathfrak n,\ad A)_{\tr}=(-c1_\mathfrak n+D,\ad A)_{\tr}=
-c\tr\ad A+(D,\ad A)_{\tr} \\
&=-c\scalar{H_0}{A}_\mathfrak s+(D,\ad A)_{\tr}
=-(\ad H_0,\ad A)_{\tr}+(D,\ad A)_{\tr}
\end{aligned}$$
(cf. \cite{Lau11:1}*{the proof of Proposition 4.2}). 
Since $(\cdot,\cdot)_{\tr}$ is nondegenerate on 
$\ad\mathfrak a_\mathfrak n$, we get 
$D=\ad H_0$ and hence 
$$\Ric^\mathfrak s=-c1_{\mathfrak s}.$$ 
Thus we obtain the following theorem, which is a particular case of  \cite{CR22:1}*{Theorem 4.1}. Also  see  \cite{Lau11:1}*{Theorem 4.8}. 
\begin{theorem} \label{th72}
Let $(\mathfrak s,\innerproduct_{\mathfrak s})$  be 
a pseudo-Iwasawa extension of a pseudo H-type Lie algebra $(\mathfrak n,\innerproduct)$. 
Then $(\mathfrak s,\innerproduct_{\mathfrak s})$ is Einstein with 
$\Ric^{\mathfrak s}=-c1_\mathfrak s$. 
\end{theorem} 
\begin{remark}
From the above results, we see that $\xi E=H_0$. 
Since $[\mathfrak a''_0,\mathfrak n_{-2}]=0$, 
for $A\in \mathfrak a''_0$ we get 
$$\tr(\ad A)=\scalar{A}{H_0}_\mathfrak s=\frac{\xi}{c}(\ad A,\ad E)_{\tr}
=-\frac{\xi}{c}\tr\ad A, $$
so $\tr\ad A=0$. 
\end{remark}
\begin{example} 
We consider the following SGLA $\gla g$ of the second kind. 
$$\begin{aligned} 
\mathfrak g&=\{~X\in\mathfrak{sl}(\ell+1,\mathbb C):
X^*S_{p,q}+S_{p,q}X=0~\}\quad  (\ell=2p+q-1,p\geqq 2),\\
\mathfrak g_{-2}
&=\{~Z=0_2\ominus 0_{\ell-3}\ominus Z_{31} :
Z_{31}\in \mathfrak{gl}(2,\mathbb C),Z_{31}^*=-K_2Z_{31}K_2~\}, \\
\mathfrak g_{-1}&=\{~X=
\begin{bmatrix} 
0 & 0 & 0 \\
X_{21} & 0 & 0 \\
0 & -K_2X_{21}^*S' & 0 \\
\end{bmatrix}:
X_{21}\in M(\ell-3,2;\mathbb C)~\}, \\
\mathfrak g_{0}& =\{~A=A_{11}\oplus A_{22}\oplus (-K_2A_{11}^*K_2)
:
A_{11}\in \mathfrak{gl}(2,\mathbb C),
A_{22}\in\mathfrak{gl}(\ell-3,\mathbb C),A_{22}^*S'+S'A_{22}=0~\}, \\
\mathfrak g_p &=\{~\transpose{X}\in\mathfrak g:X\in\mathfrak g_{-p}~\}
\quad (p=1,2),\quad \mathfrak g_p=0 \quad (|p|>2), 
\end{aligned} 
$$
where 
$$S'=K_{p-2} \ominus 1_q \ominus K_{p-2} .$$
For $B\in M(\ell-3,2,\mathbb C)$, 
we denote by $B^\natural$ the following $\ell\times\ell$ matrix: 
$$
B^\natural= 
\begin{bmatrix} 
0 & 0 & 0 \\
B & 0 & 0 \\
0 & -K_2B^*S' & 0 \\
\end{bmatrix}.$$
We define a nondegenerate symmetric bilinear form $\scalar{\cdot}{\cdot}_{-2}$ on $\mathfrak g_{-2}$ as follows: 
$$\scalar{Z}{W}_{-2}
=\frac{\eta_0}2(\tr (Z_{31})\tr(W_{31})-\tr (Z_{31}W_{31}))\quad 
(Z,W\in\mathfrak g_{-2},\eta_0=\pm1).$$
Here we assume one of the following two conditions: 
\begin{enumerate}
\renewcommand{\labelenumi}{(\roman{enumi}) }
\item $\eta_0=1$, $q=0$; 
\item $\eta_0=-1$, $p=2+2n$, $q=2m$. 
\end{enumerate}
We define a symmetric bilinear form 
$\scalar{\cdot}{\cdot}_{-1}$ by  
$$\scalar{X}{Y}_{-1}
=\eta_0\rea\tr(J_2\transpose{X_{21}}S_2Y_{21})
\quad 
(X,Y\in\mathfrak g_{-1})$$
and define a nonsingular linear operator $J_Z$ on $\mathfrak g_{-1}$ 
($Z\in\mathfrak g_{-2}$) as follows: 
$$J_Z(X)=(\transpose{S_1}\overline{X_{21}}1_{1,1}Z_{21})^{\natural}\qquad 
(X\in\mathfrak g_{-1}),$$
where 
$$\begin{aligned}
S_1&=1_{p-2,p-2}
, \quad 
S_2=J_{2p-4}\quad  (\text{if}\  \eta_0=1), \\ 
S_1&=(-1_{n,n}) \ominus(-I_{2m}) \ominus 1_{n,n}, \qquad 
S_2= J_{2n} \oplus I_{2m} \oplus(-J_{2n})
 \qquad (\text{if}\ \eta_0=-1). 
\end{aligned}$$
Note that $S'=S_1S_2$ and $S_2$ is a skewsymmetric matrix. 
Then the following conditions hold:
$$\scalar{J_Z(X)}{Y}_{-1}=\scalar{Z}{[X,Y]}_{-2},\quad 
J_Z^2=-\scalar{Z}{Z}_{-2}1_{\mathfrak g_{-1}}\quad 
(X,Y\in\mathfrak g_{-1},Z\in\mathfrak g_{-2}).$$
That is, $(\mathfrak n,\innerproduct)$ becomes a pseudo H-type Lie algebra, 
where $\mathfrak n=\mathfrak g_{-1}+\mathfrak g_{-2}$ and 
$\innerproduct=\scalar{\cdot}{\cdot}_{-1}+
\scalar{\cdot}{\cdot}_{-2}$. 
The characteristic element $E$ of $\gla g$ has the following form 
and the maximal $\mathbb R$-diagonalizable commutative subspace $\mathfrak a$ containing $E$ can be chosen as follows: 
$$E=\diag(1,1,0,\cdots,0,-1,-1),\quad 
\mathfrak a=\{~\diag(a_1,\dots,a_p,0\dots,0,-a_p,\dots,-a_1)\in\mathfrak g:
a_i\in\mathbb R~\}. $$
From \cite{Tam11:1}*{Proposition 3.8}, we see that $\dim\mathfrak a'=1$ and 
$\mathfrak a''=\mathbb RE$. 
Hence the solvmanifold $(\mathfrak s,\innerproduct_\mathfrak s)$ is Einstein in the both cases. 
\end{example}
\begin{example} 
We consider the following SGLA $\gla g$ of the second kind. 
$$\begin{aligned}
\mathfrak g&=\mathfrak{sl}(\ell+1,\mathbb R)\quad (\ell\geqq7,\ell=\text{odd}), \\
\mathfrak g_{-2}
&=\{~Z=0_2 \ominus 0_{\ell-3} \ominus Z_{31} :
Z_{31}\in \mathfrak{gl}(2,\mathbb R)~\}, \\
\mathfrak g_{-1}&=\{~X=
\begin{bmatrix} 
0 & 0 & 0 \\
X_{21} & 0 & 0 \\
0 & X_{32}  & 0 \\
\end{bmatrix}:
X_{21}\in M(\ell-3,2,\mathbb R), X_{32}\in M(2,\ell-3,\mathbb R)~\}, \\
\mathfrak g_{0}&=\{~A=A_{11}\oplus A_{22}\oplus A_{33}:
A_{11},A_{33}\in\mathfrak{gl}(2,\mathbb R),A_{22}\in\mathfrak{gl}(\ell-3,\mathbb R)~\}, \\
\mathfrak g_p &=\{~\transpose{X}\in\mathfrak g:X\in\mathfrak g_{-p}~\}
\quad (p=1,2),\quad \mathfrak g_p=0 \quad (|p|>2). 
\end{aligned}$$
We set $T=J_2\oplus J_{\ell-3}\oplus (-J_2)$ and define an involutive automorphism $\sigma$ of $\mathfrak g$ 
$$\sigma(X)=T(-\transpose{X})T^{-1}\quad X\in\mathfrak g.$$
Also we define a scalar product $\innerproduct$ on $\mathfrak g$ as follows: 
$$\scalar{X}{Y}=-\frac12\tr((\sigma(X)Y)\qquad X,Y\in\mathfrak g.$$
In particular for $Z\in\mathfrak g_{-2}$ we see that 
$\scalar{Z}{Z}=-\det(Z_{31})$. 
For $Z\in\mathfrak g_{-2}$ we define a linear operator $J_Z:\mathfrak g_{-1}\to\mathfrak g_{-1}$ by 
$$J_Z(X)=[Z,\sigma(X)]\quad X\in\mathfrak g_{-1}$$
We can prove easily 
$$\scalar{J_Z(X)}{Y}=\scalar{Z}{[X,Y]},\quad 
J_Z^2(X)=-\scalar{Z}{Z}X, $$
Thus $(\mathfrak n=\mathfrak g_{-2}\oplus\mathfrak g_{-1},\innerproduct_{\mathfrak n})$ becomes a pseudo H-type Lie algebra. 
The characteristic element of $\gla g$ is 
$$E=\diag(1,1,0,\dots,0,-1,-1).$$
We choose a maximal commutative $\mathbb R$-diagonalizable subspace $\mathfrak a$ of $\mathfrak g$ containing $E$ as follows: 
$$\mathfrak a=\{~\diag(a_1,\dots,a_{\ell+1})\in\mathfrak g:a_1,\dots,a_{\ell+1}\in\mathbb R~\}.$$
From \cite{Tam11:1}*{Proposition 3.8}, we see that $\dim\mathfrak a'=2$. 
We set 
$$L=\frac1{\ell+1}((\ell-3)1_2\oplus (-41_{\ell-3})\oplus(\ell-3)1_2);$$
then $L\in\mathfrak a''_0$ and $\mathfrak a_\mathfrak n=\mathbb RE\oplus\mathbb RL$. Hence $\dim \mathfrak a_\mathfrak n=2$. 
\end{example}
\subsection{Isometric group}
Let $(\mathfrak g,\innerproduct)$ be a metric Lie algebra 
and $(G,g)$ the pseudo-riemannian manifold attached to 
$(\mathfrak g,\innerproduct)$. 
Let $\Iso(G)$ denote the isometry group of $(G,g)$. 
The subgroup of isometries fixing the identity element denoted by $H$ 
is a closed subgroup of $\Iso(G)$ and one has 
$$\Iso(G)=G\cdot H.$$
We denote by $H^{\aut}$ the group of isometric automorphism of $G$, that is, 
$H^{\aut}=\Aut(G)\capprod H$ and set $\Iso^{\aut}(G)=G\cdot H^{\aut}$.  

Let $(\mathfrak n,\innerproduct)$ be a pseudo H-type Lie algebra and 
$(N,g)$ the pseudo-riemannian manifold attached to $(\mathfrak n,\innerproduct)$. 
The gradation of $\mathfrak n$ induces on the Lie group $N$ left invariant orthogonal distributions $\mathfrak n_{-1}N$ and $\mathfrak n_{-2}N$ such that 
$TN=\mathfrak n_{-1}N\oplus\mathfrak n_{-2}N$. 

We denote by $\Iso^{\Spl}(N)$ the group of isometries of $N$ that preserve the splitting $TN=\mathfrak n_{-1}N\oplus\mathfrak n_{-2}N$. 
Then 
$$\Iso^{\Spl}(N)=N\cdot H^{\Spl},$$
where $H^{\Spl}$ is the subgroup of isometries which preserve the splitting and fix the identity element of $N$. By \cite{BO13:1}*{Theorem 4.4},  one has 
$$\Iso(N)=\Iso^{\aut}(N)=\Iso^{\Spl}(N).$$
By Propositions \ref{prop41} and \ref{prop61}, the Lie algebra of $\Iso(N)$ 
is isomorphic to $\mathfrak n\oplus\mathfrak{so}(\mathfrak n_{-2},\innerproduct_{-2})\oplus\check{\mathfrak h}^a_0$.

\end{document}